\documentclass[11pt]{amsart}
\usepackage{graphicx}%
\usepackage{multirow}%
\usepackage{amsmath,amssymb,amsfonts}%
\usepackage{amsthm}%
\usepackage{mathrsfs}%
\usepackage{booktabs}%
\usepackage{float}
\usepackage{multirow}
\usepackage{graphicx,epsfig}
\usepackage{color}
\usepackage{epsfig,subfigure,epstopdf}
\usepackage{graphicx,euscript}
\usepackage{hyperref} 
\usepackage{multirow,array}

\def\cM{\mathcal{M}}
\def\cN{\mathcal{N}}

\def\vnu{\vec{\nu}}

\def\vy{\vec{y}}
\def\vX{\vec{X}}
\def\vY{\vec{Y}}
\def\vZ{\vec{Z}}
\def\vmY{\vec{\mathcal{Y}}}
\def\vmZ{\vec{\mathcal{Z}}}
\def\hX{\hat{X}}

\def\vxi{\vec{\xi}}

\def\Id{{\rm Id}}
\def\vv{\vec{v}}
\def\cI{\mathcal{I}}
\def\vphi{\vec{\varphi}}
\def\half{1/2}

\def\ve{\vec{e}}
\def\tvnu{\vec{{\nu}}}

\def\vtau{\vec{\tau}}

\def\ttau{\vec{\tau}}

\def\vd{\vec{d}}

\def\vb{\vec{b}}
\def\vrho{\vec{\rho}}

\def\Ritz{\mathcal{R}}
\def\R{\mathcal{R}^{\mathcal{I}}}
\def\mH{\mathcal{H}}

\def\mI{\mathbb{I}}
\def\hRh{\hat{R}_h}
\def\hrho{{\hat{\rho}}}

\theoremstyle{thmstyletwo}%
\newtheorem{theorem}{Theorem}[section]
\newtheorem{remark}{Remark}[section]%
\newtheorem{lemma}{Lemma}[section]
\theoremstyle{definition}
\newtheorem{example}{Example}[section]%
\numberwithin{equation}{section}

\usepackage{cleveref}

\begin{document}


\title[ ]{A revisit to DeTurck method on curve shortening flow with optimal error analysis}
	\author{Beiping Duan}
	\address{MSU-BIT-SMBU Joint Research Center of  Applied Mathematics, 
	Shenzhen MSU-BIT University, Shenzhen  518172, P.R. China}
\email{duanbeiping@smbu.edu.cn}

\begin{abstract}
The curve shortening--DeTurck flow introduced by Elliott and Fritz (IMA J. Numer. Anal. 37(2): 543--603, 2017) is a reparameterization of the curve shortening flow via the DeTurck trick. For corresponding discrete schemes, although the $H^1$-error estimate has been established, the optimal $L^2$-error estimate remains open due to technical difficulties. In this paper, we prove optimal $L^2$-error estimates for linearized Euler and Crank--Nicolson discretizations combined with finite elements of degree $k\ge 1$ in space. Moreover, we provide an extrinsic approach to derive the curve shortening--DeTurck flow. Numerical experiments are presented to illustrate the mesh distribution properties and to verify the convergence rates in both space and time.
\end{abstract}



\maketitle

\section{Introduction}
Suppose $\Gamma(t)$ is a closed regular planar curve  which evolves with time and is parameterized by the diffeomorphism $\vX(t,s): I\rightarrow \Gamma(t)$ with  $I$  a periodic interval. The evolving curve $\Gamma(t)$ then can be described by virtue of the flow map $\vX(t,s)$:
\begin{equation}\label{eqn-evolving}
	\frac{\partial\vX(t,s)}{\partial t}=\vv\circ \vX(t,s),
\end{equation}
where in the case of curve shortening flow (CSF) $\vv=\partial^2_{\gamma}\Id|_{\Gamma(t)}$ with  $\Id$ the identity map from $\Gamma(t)$ onto itself and $\gamma$ its arc length parameter. If we introduce the perimeter $L=\int_{\Gamma(t)}1\, d\gamma$ as the functional energy, then the CSF   can be considered as the $L^2$  gradient flow.

Utilizing the chain rule one can reform the CSF in the parameter space $I$ base on which the weak form reads: Find $\vX(t,s)\in H^1(I)^2$ with $\vX(0,s)=\vX^0(s)$ such that 
\begin{equation}\label{ori-weak}
	\int_I \frac{\partial\vX}{\partial t} \cdot \vxi\; \big|\partial_s\vX\big|ds=\int_I \frac{\partial_s\vX\cdot \partial_s\vxi}{\big|\partial_s\vX\big|} ds=0 \quad \forall\vxi\in H^1(I)^2.
\end{equation}
Applying linearized Euler scheme in time (dealing with $|\partial_s\vX|$ explicitly) and piecewise linear elements in space one can get Dziuk's original method \cite{Dziuk-1990}. 

The first attempt at an error analysis of Dziuk's scheme was carried out in \cite{dziuk1994convergence}, where the \( H^1 \) error of a semidiscrete formulation for the CSF was studied. The \( H^1 \) error for the fully discrete scheme was established by  \cite{ye2021convergence}. We also note that prior to their work,  \cite{li2020convergence} derived the \( L^2 \)-error estimate for the fully discrete scheme using finite elements of degree \( k \geq 3 \). Subsequently,  \cite{li2021convergence,bai2023erratum} proved the convergence of the semidiscrete scheme for closed surfaces using finite elements of degree \( k \geq 6 \), where the matrix-vector formulation developed in \cite{KLLP2017, Kovacs-Li-Lubich-2019} was employed.

Dziuk's method is simple to implement and leads to a decoupled system in the \(x\)- and \(y\)-directions (and \(z\)-direction if applicable). However, the drawback is that it includes only normal velocities, which may cause vertices on the numerical curve to collapse into each other, so that \(\big|\partial_s \vX_h^n\big|\) — the discrete analogue of \(\big|\partial_s\vX(t)\big|\) at \(t=t_n\) — approaches \(0\). For the CSF, an early method incorporating tangential motion was presented by \cite{DeckelnickDziuk}, where the weak form reads: Find \(\vX(t,s)\in H^1(I)^2\) with \(\vX(0,s)=\vX^0(s)\) such that 
\begin{equation}\label{dd-weak}
	\int_I \frac{\partial\vX}{\partial t} \cdot \vxi\; \big|\partial_s\vX\big|^2 ds+\int_I {\partial_s\vX\cdot \partial_s\vxi}ds = 0 \quad \forall\vxi\in H^1(I)^2.
\end{equation}
The spatial \(H^1\) error of the semidiscrete scheme for \eqref{dd-weak} was also given in their paper. To the best of our knowledge, the optimal \(L^2\) error of a fully discrete scheme based on \eqref{dd-weak} was first proved by \cite{deckelnick2025second}, published 30 years after \cite{DeckelnickDziuk}. Some of the tricks they used in their proof provided the original inspiration for the error estimates in this paper.

\cite{BGN07a} presented a framework (referred as the Barrett-Garcke-N\"urnberg, BGN scheme) for inducing tangential motions for \eqref{eqn-evolving} with $\vv\cdot \vnu=f(H)$, where $H$ and $\vnu$ are the curvature and the unit normal, respectively. When applied to the CSF, this method is linear and preserves the energy-decaying property. Moreover, it also enjoys good mesh distribution and has been widely used since then. Later, the idea was extended to anisotropic case in \cite{barrett2008numerical}. However, the convergence of the BGN method remains open to date. The most closely related work in this direction is that of \cite{bai_convergence_2024}, in which the optimal $L^2$ error for a stabilized BGN scheme using finite elements of degree $k\ge 2$ was established.

The DeTurck method that we will discuss in this paper was originally proposed by \cite{Elliott-Fritz-2017} for the CSF and the mean curvature flow (MCF). The idea is to use the DeTurck trick given in \cite{deturck1983deforming} to reparametrize the original flow so that the new flow, referred as the curve shortening-/mean curvature-DeTurck flow, includes a proper tangential velocity, which can maintain the mesh quality produced by the corresponding numerical scheme. An optimal $H^1$-error estimate for the semidiscrete scheme of the CSF was given. The approach they adopted to obtain the estimate is similar to that used in \cite{DeckelnickDziuk}. It is also worth pointing out that \eqref{dd-weak} is actually a special case of the DeTurck method.

For error estimates on the curve evolution, many efforts have also been made beyond the CSF, see, e.g., \cite{deckelnick2003error,deckelnick2009error,barrett2017numerical,pozzi2023convergence}, all of which focus on semidiscrete schemes in space and optimal $H^1$ estimates. Recently, \cite{deckelnick2025finite} demonstrated optimal $L^2$ estimates for the semidiscrete schemes of the curve diffusion and the elastic flow. Their formulations can induce tangential motions that lead to an equidistributed mesh in practice. For earlier work on the convergence analysis of fully discrete schemes for the elastic flow, one may refer to \cite{bartels2013simple}, in which piecewise B\'ezier curves were used for the spatial discretization.

The harmonic map between two surfaces is often employed to approximate the conformal map and, in the discrete setting, helps preserve mesh quality. Building on this concept, we introduced harmonic maps and their heat flows to develop mesh-preserving parametric FEMs for geometric flows of surfaces; see \cite{duan2024new,duan2024energy,duan2025mesh}. In fact, it was shown in \cite{duan2024energy} that the variant DeTurck method for the MCF of surfaces derived via the intrinsic approach in \cite{Elliott-Fritz-2017} can be obtained by coupling the original flow with a harmonic map heat flow from a reference surface onto the unknown with a specific diffusion coefficient.

In this paper (Section \ref{sec-harmonic}) we will show that the aforementioned coupling approach can be used to derive the DeTurck method for the CSF. In fact, the harmonic map between curves acts by stretching or shrinking the curve uniformly at each point. When applied to discrete schemes, this results in the target mesh distributing its vertices in a manner that resembles the reference mesh — specifically, the ratios between corresponding (curved) edge lengths of the target and reference curves are equal. This property explains why discrete schemes based on \eqref{dd-weak} can distribute the mesh vertices equidistantly as time evolves. We would also like to mention that the harmonic (Dirichlet) energy was introduced in \cite{pozzi2023convergence,deckelnick2025finite}  to penalize the length of the evolving curve under the elastic flow.

The main contribution of this paper is the optimal \(L^2\)-error estimates for two fully discrete DeTurck methods — the linearized Euler and Crank–Nicolson schemes — presented in Section \ref{sec-Euler} and Section \ref{sec-CN}, respectively. The key to obtaining the optimal estimates is to establish the superconvergence of the numerical solutions, namely, the \(H^1\) seminorm satisfies \(\big\|\partial_s(\vX_h^n-\R_h \vX(t_n))\big\| \le \mathcal{O}((\Delta t)^l +h^{k+1})\) with \(l=1\) under Euler discretization and \(l=2\) for the Crank–Nicolson scheme. Here \(\R_h\) denotes the interpolation or the Ritz projection,  depending on the polynomial degree \(k\). We provide some numerical tests in Section \ref{sec-exp} to verify our theoretical analysis and the mesh distribution.

\section{Revisit to the  DeTurck method}\label{sec-harmonic}
\subsection{Harmonic map heat flows between curves}
Denote by $\cM$ and $\cN$ two regular curves, given by the parameterizations $\vphi_1(s)$ and $\vphi_2(\rho)$, respectively. For convenience we take $s$ and $\rho$ as the arc-length parameters of $\cM$ and $\cN$, respectively.  For any $C^2$-diffeomorphism $\Psi$, the harmonic energy $\mH[\Psi]$ is defined by 
\begin{equation}\label{Dirichlet-Energy}
	\mH[\Psi]=\frac{1}{2}\int_s\left|\partial_s(\Psi\circ\vphi_1)\right|^2 ds.
\end{equation} 
We introduce the map between the parameter spaces of $\cM$ and $\cN$: $\vphi_2^{-1}\circ\Psi\circ\vphi_1$, and denote it by $f$ for simplicity.  Using the identity $\Psi\circ\vphi_1=\vphi_2\circ f$ and the chain rule, it follows
\begin{equation*}
	\left|\partial_s(\Psi\circ\vphi_1)\right|^2=\left|\partial_\rho\vphi_2\right|^2\left(\partial_sf\right)^2=\left(\partial_sf\right)^2,
\end{equation*}
where the last step follows from that $\rho$ is the arc length parameter. So we can put the harmonic energy in the following intrinsic form:
\begin{equation}\label{Dirichlet-Energy-in}
	\mH[\Psi]=\frac{1}{2}\int_s \left(\partial_sf\right)^2  ds:=\mH[f].
\end{equation}
We say $\Psi$ is harmonic if the variational derivative $\frac{\delta \mH[f]}{\delta f}=0$.
\begin{lemma}
	Suppose $\cM$ and $\cN$ are two regular curves with $|\cM|$ and $|\cN|$ the corresponding perimeters.  Let $\Psi$ be a $C^2$-diffeomorphism from $\cM$ onto $\cN$, then $\Psi$ is harmonic implies $\frac{d\rho}{d s}=\frac{|\cN|}{|\cM|}$.
\end{lemma}
\begin{proof}
	It is easy to obtain $\frac{\delta E[f]}{\delta f}=0$ gives $\partial_{ss}f=0$. Without losing generality we assume that $f(0)=0$ and $f(|\cM|)=|\cN|$, together with $\partial_{ss}f=0$ then the desired equality follows.
\end{proof}

The lemma above reveals that harmonic maps between curves actually stretch or shrink the curve uniformly at each point. The extrinsic version of the harmonic map heat flow between $\cM$ and $\cN$ can be given by
\begin{equation}\label{heat-flow-c}
	\partial_t \Psi= \left[\partial_{ss}  \Psi \cdot(\vtau\circ \Psi)\right](\vtau\circ\Psi),
\end{equation}
where and hereafter, following the convention,  we simply use e.g., $\partial_{s} \Psi$ instead of $\partial_{s} (\Psi\circ\vphi_1)$ when $s$ is the  arc length parameter. In fact, one can get \eqref{heat-flow-c} from the intrinsic version $\partial_t f=\partial_s^2 f$ with utilizing the chain rule.  The heat flow \cref{heat-flow-c} decreases the harmonic energy and finally approaching the equilibrium that satisfies $\frac{d\rho}{d s}=\frac{|\cN|}{|\cM|}$.

\subsection{Reformulations for curve shortening flow}
Now we focus on the target equation \cref{eqn-evolving} with $\vv=\partial_\gamma^2\Id|_{\Gamma(t)}$. Without losing generality we adopt the unit circle centered at the origin as $\cM$ and  let $I=\mathbb{R}/2\pi$ be the periodic interval.  Furthermore, let $\vtau$ represent the unit tangent of the unknown $\Gamma(t)$, and for notational simplicity, we also adopt $\vtau$ and $\vnu$ to represent $\vtau\circ\vX$ and  $\vnu\circ\vX$, respectively.

Follow the idea from \cite{duan2024energy} we reformulate the original flow by coupling it with the harmonic map heat flow form $\cM$ onto $\Gamma(t)$. Note that $I$ is actually the parameter space of the arc length, the map between $\cM$ to $\Gamma(t)$ is equivalent to the map from $I$ to $\Gamma(t)$.  Consequently, the evolution governed by the CSF can be described by the following system, in which  for simplicity we still denote the new flow map by $\vX(t,s)$:
\begin{equation}\label{CE-HMHF}
	\frac{\partial \vX}{\partial t} =\partial_\gamma^2\Id \circ \vX+  k(t,s)(\partial_{s}^2\vX \cdot  \ttau)\ttau,
\end{equation}
where $s\in I$ and $ k(t,s): I \rightarrow \mathbb{R}_{> 0}$ is a smooth scalar function. Recall that $\partial_\gamma^2\Id=-H\vnu$ with $H$ the curvature and $\vnu$ the unit outward normal vector. So by testing  \eqref{CE-HMHF} with $(\tau\otimes\tau)\,\partial_t\vX$, moving the diffusion coefficient $k(t,s)$ to the left hand side and integrating both sides on $I$ one can get
\begin{equation}\label{dacy-harmonic}
	\int_I k^{-1}\big(\partial_t\vX\cdot\vtau\big)^2ds=\int_I (\partial_{s}^2\vX \cdot  \ttau)\ttau\cdot\partial_t\vX ds=-\frac{d}{dt}\mH[\vX].
\end{equation} 
Furthermore, test \eqref{CE-HMHF} by $(\vnu\otimes\vnu)\,\partial_t\vX$ and integrate along $\Gamma(t)$ then it follows
\begin{equation}\label{dacy-energy}
	\int_{\Gamma(t)} (\partial_t\vX\cdot\vnu)^2 ds=\int_{\Gamma(t)} \partial_\gamma^2 \Id \cdot\partial_t\vX d\gamma=-\frac{d}{dt} L[\vX],
\end{equation}
where $L[\vX]=\int_{I}|\partial_s\vX| ds$ is the perimeter of $\Gamma(t)$. The two equations \eqref{dacy-harmonic} and \eqref{dacy-energy} imply the reformulated system \eqref{CE-HMHF} decreases both the harmonic energy and the perimeter.

Next, we will show that the curve shortening-DeTurck flow is a special case of \eqref{CE-HMHF}. Note that 
\begin{equation}\label{partial_ID}
	\begin{aligned}
		\partial_\gamma^2\Id \circ \vX&=\frac{1}{|\partial_s X|}\left(\partial_s\Big(\frac{\partial_s \vX}{|\partial_s \vX|}\Big)\right)=\frac{1}{|\partial_s \vX|^2}\partial_s^2\vX-\frac{1}{|\partial_s \vX|^4}(\partial_s^2\vX\cdot\partial_s\vX)\partial_s\vX\\
		&=\frac{1}{|\partial_s \vX|^2}\partial_s^2\vX-\frac{1}{|\partial_s \vX|^2}(\partial_s^2\vX\cdot\ttau)\ttau,
	\end{aligned}
\end{equation}
so we can put \eqref{CE-HMHF} into the following form
\begin{equation}\label{CE-HMHF-1}
	|\partial_s X|^2\frac{\partial \vX}{\partial t} =\partial_s^2\vX+  \big(k(t,s)|\partial_s X|^2-1\big)(\partial_{s}^2\vX \cdot  \ttau)\ttau.
\end{equation}
Take $k(t,s)=|\partial_s \vX|^{-2}$ in \eqref{CE-HMHF-1} then it reduces to 
\begin{equation}\label{Deckelnick-Dziuk}
	|\partial_s X|^2\frac{\partial \vX}{\partial t} =\partial_s^2\vX,
\end{equation}
which, to ours best knowledge, appeared firstly in \cite[eq.(5)]{DeckelnickDziuk}. Take $k=\alpha^{-1}|\partial_s \vX|^{-2}$ with $\alpha$ a positive constant then \eqref{CE-HMHF-1} reduces to
\begin{equation}\label{Elliott-Fritz}
	|\partial_s X|^2\frac{\partial \vX}{\partial t} =\partial_s^2\vX+(\alpha^{-1}-1)(\partial_s^2\vX \cdot  \ttau)\ttau,
\end{equation}
which was proposed by \cite[eq. (2.16)]{Elliott-Fritz-2017} based on the DeTurck trick. Note that $\partial_{s}^2 \vX \cdot\ttau=\partial_s|\partial_s\vX|$, then the tangential term in \eqref{Elliott-Fritz} can be further put into the following form
\begin{equation*}
	\begin{aligned}
		(\partial_s^2\vX\cdot\ttau)\ttau=\left(\partial_s|\partial_s\vX|\right)\ttau&=\partial_s\left(|\partial_s\vX|\ttau\right)-|\partial_s\vX|\partial_s\ttau\\
		&=\partial_s^2\vX-|\partial_s\vX|^2\partial_\gamma^2\Id\circ\vX\\
		&=\partial_s^2\vX-|\partial_s\vX|^2(\tvnu\otimes\tvnu)\frac{\partial\vX}{\partial t}.
	\end{aligned}
\end{equation*}
Substitute it into \eqref{Elliott-Fritz} it follows
\begin{equation}\label{Elliott-Frit-1}
	|\partial_s X|^2\left(\alpha\mI+(1-\alpha)\tvnu\otimes\tvnu\right)\frac{\partial \vX}{\partial t} =\partial_s^2\vX,
\end{equation}
which is exactly the one appeared in \cite[eq. (2.21)]{Elliott-Fritz-2017}, where  $\mI$ denotes the identity matrix. 

\begin{remark}
	As was reported in  \cite[Figure 1]{deckelnick2025second} and \cite[Example 2]{Elliott-Fritz-2017}, discrete version of \eqref{Deckelnick-Dziuk} or \eqref{Elliott-Frit-1} owns good property on  equidistribution of the mesh. This is because the corresponding reformulated system is a special case of \eqref{CE-HMHF} that decreases the harmonic energy as time evolves. So it leads to equidistributed mesh asymptotically when $I$ is uniformly divided, as long as the scheme is convergent. In addition, following \cite{Elliott-Fritz-2017} we refer both \eqref{CE-HMHF-1} and \eqref{Elliott-Frit-1} the curve shortening-DeTurck flow. 
\end{remark}

With this choice, our weak formulation for the CSF evolving in $\mathbb{R}^2$ reads: Find $\vX(t,s)\in H^1(I)^2$ with $\vX(0,s)=\vX^0(s)$ such that
\begin{equation}\label{weak-form}
	\left(|\partial_s X|^2\left(\alpha\mI+(1-\alpha)\tvnu\otimes\tvnu\right)\frac{\partial \vX}{\partial t} ,\vxi\right)+\left(\partial_s\vX,\partial_s\vxi\right)=0 \quad \forall\vxi\in H^1(I)^2,
\end{equation}
where $\alpha> 0$ is a constant along the curve.

\section{Error estimate for Euler scheme}\label{sec-Euler}
Let $\cup_{j=1}^N I_j$ be the partition of $I$ with $I_j=[s_{j-1}, s_j]$. Hereafter, we always use the periodicity of the index, namely, $-1$ corresponds to $N-1$ and $N+1$ corresponds to $1$. Based on the partition  we define the following finite element space
\begin{equation*}
	S(I)=\left\{w|\; w \mbox{ is continuous  and } w|_{I_j} \mbox{ is polynomials of degree } k \right\}.
\end{equation*}

Suppose that $\vX_h^0\in S(I)^2$ approximates $\vX(0)$ in an appropriate manner. The semidiscrete weak formulation to approximate $\vX(t)$ reads: Find $\vX_h(t)\in S(I)^2$ with $\vX_h(0)=\vX_h^0$ such that 
\begin{equation}\label{semi-weak}
	\left(|\partial_s \vX_h|^2\left(\alpha\mI+(1-\alpha)\tvnu_h\otimes\tvnu_h\right)\frac{\partial \vX_h}{\partial t} ,\vxi_h\right)+\left(\partial_s\vX_h,\partial_s\vxi_h\right)=0 \quad \forall \vxi_h\in S(I)^2,
\end{equation}
where $\vnu_h$ is the unit normal given by $\vnu_h=\vtau_h^{\perp}:=(\partial_s\vX_h)^\perp/|\partial_s\vX_h|$ with ${}^\perp$ denoting the clockwise rotation by $\frac{\pi}{2}$. Note that $\vnu_h\otimes \vnu_h=\mI-\vtau_h\otimes \vtau_h$ for planar curves one can reform \eqref{semi-weak} into the following form
\begin{equation}\label{semi-weak-1}
	\left( \big(|\partial_s\vX_h|^2\mI +(\alpha-1)\partial_s \vX_h\otimes \partial_s\vX_h\big)\frac{\partial \vX_h}{\partial t} ,\vxi_h\right)+\left(\partial_s\vX_h,\partial_s\vxi_h\right)=0\quad \forall \vxi_h\in S(I)^2.
\end{equation}

We divide $[0,T]$ into $M$ uniform subintervals and denote by $t_n=n\Delta t$ with $\Delta t=\frac{T}{M}$. Let $\vX_h^{n}$ be the approximate solution of $\vX(t_n)$ in our fully discrete scheme, then our linearized Euler scheme reads: Find $\vX_h^{n+1}\in S(I)^2$ such that
\begin{equation}\label{Euler}
	\left(G(\vX_h^n)\frac{\vX_h^{n+1}-\vX_h^n}{\Delta t},\vxi_h\right)+\left(\partial_s\vX_h^{n+1},\partial_s\vxi_h\right)=0 \quad \forall \vxi_h\in S(I)^2,
\end{equation}
where for any $\vb\in H^1(I)^2$ we define
\begin{equation}\label{eqn-G}
	G(\vb)=\big|\partial_s\vb\big|^2\mI +(\alpha-1)\partial_s\vb\otimes\partial_s\vb.
\end{equation}

For  $\vy(s)\in H^1(I)^2$ we introduce the Ritz projection $\Ritz_h$, given by: 
\begin{equation*}
	\left(\partial_s(\Ritz_h \vy- \vy),\partial_s\vxi_h\right)=0 \quad \forall\vxi_h\in S(I)^2 \quad \mbox{with }\int_I \Ritz_h\vy-\vy\,ds=\vec{0}.
\end{equation*}
Then it is known that for smooth enough $\vy$
\begin{equation*}
	\big\| \Ritz_h\vy-\vy\big\| + h\big\| \partial_s \Ritz_h\vy-\partial_s\vy\big\|\le ch^{k+1} \big\| \partial_s^{k+1}\vy\big\|,
\end{equation*}
where $\| \cdot \|$ represents the $L^2(I)$ norm. Here and hereafter we always use $c$ or $\tilde{c}$ to denote a generic positive constant which is independent of $h$, $\Delta t$ and $\alpha$, and may be different in different places. 

Denote by $\cI_h$ the piecewise interpolation operator mapping from $C(I)$ onto $S(I)$, then it is easy to verify in the case of $k=1$
\begin{equation*}
	\left(\partial_s(\cI_h \vy-\vy),\partial_s\vxi_h\right)=0 \quad  \forall \vxi_h\in S(I)^2.
\end{equation*}
Besides, let $\R_h$ represent the Ritz projection operator when $k\ge 2$, and in the case of $k=1$ we use it to denote the interpolation operator $\cI_h$. With this we split the error $\vX(t_n)-\vX_h^{n}$ into two parts
\begin{equation*}
	\ve_h^{n}=\R_h\vX(t_n)-\vX_h^n \quad \mbox{and}\quad \vrho_h=\R_h\vX(t)-\vX(t),
\end{equation*}
so  $\vX(t_n)-\vX_h^{n}=\ve_h^n-\vrho_h^n$ with $\vrho_h^n=\vrho_h(t_n)$. We initialize the scheme by set  $\vX_h^0=\R_h\vX(0)$.

\begin{theorem}\label{est-Euler-H1}
	Suppose \eqref{Elliott-Frit-1} has a unique solution $\vX(t,s)$ which is smooth enough with respect to both $t$ and $s$ up to some $t>T$. Then there exists  some $h_0\le 1$ such that for $h\le h_0$ and $t\le h^{1/2}$ being small enough, our Euler scheme \eqref{Euler} has a unique solution and satisfies
	\begin{equation}\label{est-Euler}
		\big\|\ve_h^m \big\| + \big\|\partial_s\ve_h^{m}\big\|\le c(\alpha)\big(\Delta t+h^{k+1}\big) \quad \mbox{for }m \Delta t \le T,
	\end{equation}
	where $ c(\alpha)$ is a constant that is independent of $\Delta t$ and $h$ but depends on $\alpha$. Together with the estimate for $\vrho_h$ we have the optimal error bound
	\begin{equation}\label{est-total}
		\big\|\partial_s^p(\vX_h^m-\vX(t_m)) \big\| \le   c(\alpha)\big(\Delta t +h^{k+1-p}\big),\quad p=0,1.
	\end{equation}
\end{theorem}
\begin{proof}
	We only demonstrate the estimate for $\ve_h^m$.  To this end we will show that by choosing proper $K_\alpha$ and $\beta_\alpha$, the following bound holds for small enough $h$ and $\Delta t$
	\begin{equation}\label{est-Euler-R}
		\left\|\partial_s\ve_h^{m}\right\|^2\le K_\alpha^2e^{\beta_\alpha t_{m-1}}((\Delta t)^2+h^{2k+2}),
	\end{equation}
	with $K_\alpha$ and $\beta_\alpha$ two constants independent of $h$ and $\Delta t$ but dependent on $\alpha$.  Suppose $|\partial_s\vX|$ is bounded from both above and below, then by the  estimate 
	\begin{equation*}
		\|\partial_s(\vX-\R_h\vX)\|\le c h^{k}\|\partial_s^{k+1}\vX\|
	\end{equation*}
	and the inverse estimate, we know there exists some $h_0$ such that for $h\le h_0$
	\begin{equation*}
		c_1\le |\partial_s\R_h\vX(t)|\le c_2\quad \mbox{holds }\forall t\le T.
	\end{equation*}
	
	We shall prove \eqref{est-Euler-R} by mathematical induction. Note that $\vX_h^0=\R_h\vX(0)$, so $m=0$ is obviously satisfied. Suppose $m\le n$ is verified.  Applying the triangle inequality and the inverse estimate we have
	\begin{equation*}
		\begin{aligned}
			|\partial_s\vX_h^n|&\le |\partial_s\R_h\vX(t_n)|+|\partial_s \R_h\vX(t_n)-\partial_s\vX_h^{n}|\\
			& \le c_2+ c h^{-1/2}\big\|\partial_s \R_h\vX(t_n)-\partial_s\vX_h^{n}\big\|\\
			&\le c_2+cK_\alpha e^{\frac{\beta_\alpha}{2} T}(h^{-1/2}\Delta t+h^{k+1/2})
		\end{aligned}
	\end{equation*}
	and
	\begin{equation*}
		\begin{aligned}
			|\partial_s\vX_h^n|&\ge |\partial_s\R_h\vX(t_n)|-|\partial_s \R_h\vX(t_n)-\partial_s\vX_h^{n}|\\
			& \ge c_1- c h^{-1/2}\big\|\partial_s \R_h\vX(t_n)-\partial_s\vX_h^{n}\big\|\\
			&\ge c_1-cK_\alpha e^{\frac{\beta_\alpha}{2}T}(h^{-1/2}\Delta t+h^{k+1/2}).
		\end{aligned}
	\end{equation*}
	So for $h\le h_0$ and
	\begin{equation}\label{cond-1}
		c K_\alpha e^{\frac{\beta_\alpha}{2} T}(h^{-1/2}\Delta t+h^{k+1/2})\le c_1/2,
	\end{equation}
	we can conclude
	\begin{equation*}
		\frac{c_1}{2}\le |\partial_s\vX_h^n|\le2c_2.
	\end{equation*}
	Rewrite \eqref{Euler} into the following form
	\begin{equation}\label{re-Euler}
		A(\vX_h^{n+1},\vxi_h):=\left(G(\vX_h^n){\vX_h^{n+1}},\vxi_h\right)+\Delta t \left(\partial_s\vX_h^{n+1},\partial_s\vxi_h\right)=\left(G(\vX_h^n){\vX_h^{n}},\vxi_h\right).
	\end{equation}
	Note that for $s\in (s_{j-1},s_j)$ the $2$ eigenvalues of $G(\vX_h^n)$ are $|\partial_s\vX_h^n|^2$ and $\alpha|\partial_s\vX_h^n|^2$.  Together with the lower and upper bounds for $|\partial_s\vX_h^n|$,  we can conclude the bilinear form $A(\cdot,\cdot)$ is continuous and coercive on $S(I)^2$. The existence and uniqueness of $\vX_h^{n+1}$ follows directly. Next we turn to the error estimates.
	
	Denote by $G^{n}:=G(\vX(t_n))$ and rewrite \eqref{weak-form} at $t=t_{n+1}$ into the following form for $\vxi=\vxi_h\in S(I)^2$
	\begin{equation}\label{re-weak-form}
		\left(G(\vX_h^{n})\delta_t\R_h\vX(t_{n+1}),\vxi_h\right)+\left(\partial_s\R_h\vX(t_{n+1}),\partial_s\vxi_h\right)=\left(\vd^{n+1},\vxi_h\right),
	\end{equation}
	where $\delta_t\R_h\vX(t_{n+1}):=\frac{1}{\Delta t}(\R_h\vX(t_{n+1})-\R_h\vX(t_n))$. The defect term yields 
	\begin{equation*}
		\begin{aligned}
			\left(\vd^{n+1},\vxi_h\right)=& \left(G(\vX_h^n)\delta_t\R_h\vX(t_{n+1})-G^{n+1}\partial_t\vX(t_{n+1}),\vxi_h\right)\\
			&+\left(\partial_s\R_h\vX(t_{n+1})-\partial_s\vX(t_{n+1}),\partial_s\vxi_h\right)\\
			=&\left(G(\vX_h^n)\delta_t\R_h\vX(t_{n+1})-G^{n+1}\partial_t\vX(t_{n+1}),\vxi_h\right),
		\end{aligned}
	\end{equation*}
	where we utilized the property of $\R_h$. 
	Combining \eqref{re-weak-form} with \eqref{Euler} it follows
	\begin{equation}\label{eq-Error}
		\left(G(\vX_h^{n})\delta_t\ve_h^{n+1},\vxi_h\right)+\left(\partial_s\ve_h^{n+1},\partial_s\vxi_h\right)=\left(\vd^{n+1},\vxi_h\right),
	\end{equation}
	where $\delta_t\ve_h^{n+1}:=\frac{\ve_h^{n+1}-\ve_h^n}{\Delta t}$. Take $\vxi_h=\ve_h^{n+1}-\ve_h^n$ then it follows
	\begin{equation}\label{est-1}
		\begin{aligned}
			\frac{1}{\Delta t}\left(G(\vX_h^{n})(\ve_h^{n+1}-\ve_h^n),\ve_h^{n+1} -\ve_h^n\right)&+\frac{1}{2}\left\|\partial_s(\ve_h^{n+1}-\ve_h^n) \right\|^2+\frac{1}{2}\left\|\partial_s\ve_h^{n+1}\right\|^2\\ &=\frac{1}{2}\left\|\partial_s\ve_h^{n}\right\|^2 +\left(\vd^{n+1},\ve_h^{n+1}-\ve_h^n\right).
		\end{aligned}
	\end{equation}
	The first term on the left hand side (LHS) yields
	\begin{equation}\label{est-left}
		\frac{1}{\Delta t}\left(G(\vX_h^{n})(\ve_h^{n+1}-\ve_h^n),\ve_h^{n+1} -\ve_h^n\right)\ge \frac{c_\alpha c_1^{2}}{4\Delta t}\left\|\ve_h^{n+1}-\ve_h^n\right\|^2
	\end{equation}
	with $c_\alpha=\min(1,\alpha)$. Next we focus on the defect term. We split it into the following form
	\begin{equation*}
		\begin{aligned}
			\left(\vd^{n+1},\ve_h^{n+1}-\ve_h^n\right)=&\left(G(\vX_h^n)\delta_t \vrho_h(t_{n+1}),\ve_h^{n+1}-\ve_h^n\right)\\
			&+\left(G(\vX_h^n)\delta_t\vX(t_{n+1})-G(\vX_h^n)\partial_t\vX(t_{n+1}),\ve_h^{n+1}-\ve_h^n\right)\\
			&+\left((G(\vX_h^n)-G^{n})\partial_t\vX(t_{n+1}),\ve_h^{n+1}-\ve_h^n\right)\\
			&+\left((G^n-G^{n+1})\partial_t\vX(t_{n+1}),\ve_h^{n+1}-\ve_h^n\right):=d_{1}+d_{2}+d_{3}+d_{4}.
		\end{aligned}
	\end{equation*}
	By the mean value theorem we know there exists some $t_{n+\theta}$ with $\theta\in (0,1)$ such that
	\begin{equation*}
		\vrho_h(t_{n+1})-\vrho_h(t_n)=\partial_t\vrho_h(t_{n+\theta})\Delta t.
	\end{equation*}
	So appealing to the estimate for $\R_h$ with noting $\max(1,\alpha)\le 1+\alpha$ we have
	\begin{equation}\label{d1}
		\begin{aligned}
			d_{1}&\le \big\|G^{1/2}(\vX_h^n)\partial_t \vrho_h(t_{n+\theta})\big\|  \big\|G^{1/2}(\vX_h^n)(\ve_h^{n+1} -\ve_h^n)\big\|\\
			&\le 4c^2_2(1+\alpha)\big\|\partial_t \vrho_h(t_{n+\theta})\big\|  \big\| \ve_h^{n+1} -\ve_h^n\big\|\\
			&\le c\, c^2_2(1+\alpha) h^{k+1} \big\| \ve_h^{n+1} -\ve_h^n\big\|\\
			&\le \frac{c_\alpha c_1^2}{\Delta t}\epsilon \big\| \ve_h^{n+1} -\ve_h^n\big\|^2+\frac{c(1+\alpha)^2 h^{2k+2}}{c_\alpha\epsilon}\Delta t,
		\end{aligned}
	\end{equation}
	where we utilized $\R_h$ commutes with the time derivative. For $d_2$, Taylor expansion gives
	\begin{equation}\label{d2}
		\begin{aligned}
			d_{2}&\le  \big\|G^{1/2}(\vX_h^n) (\delta_t\vX-\partial_t\vX)(t_{n+1})\big\|  \big\|G^{1/2}(\vX_h^n)(\ve_h^{n+1} -\ve_h^n)\big\|\\
			&\le 4c_2^{2}(1+\alpha) \big\|(\delta_t\vX-\partial_t\vX)(t_{n+1})\big\|  \big\|\ve_h^{n+1} -\ve_h^n\big\| \\
			&\le  c\, c_2^{2}(1+\alpha) \Delta t  \big\|\ve_h^{n+1} -\ve_h^n\big\| \\
			&\le \frac{c_\alpha c_1^2}{\Delta t} \epsilon\big\| \ve_h^{n+1} -\ve_h^n\big\|^2+\frac{c\,(1+\alpha)^2 }{c_\alpha\epsilon}(\Delta t)^3.
		\end{aligned}
	\end{equation}
	For $d_{4}$, it follows from the smoothness of $G(\vX(t))$ in time 
	\begin{equation}\label{d_4}
		d_{4}\le c (1+\alpha)\Delta t\big\|\ve_h^{n+1}-\ve_h^n\big\|\le \frac{c_\alpha c_1^2}{\Delta t} \epsilon\big\| \ve_h^{n+1} -\ve_h^n\big\|^2+\frac{c\,(1+\alpha)^2}{c_\alpha\epsilon}(\Delta t)^3,
	\end{equation}
	where and hereafter, to keep the constants involving $\alpha$ concise,  we always use the fact that any linear combination of $\alpha$ and a constant can be bounded by $ c(1+\alpha)$.
	
	Next let us focus on the term $d_{3}$. We split it into  the following terms
	\begin{equation*}
		\begin{aligned}
			d_{3}=&\left(\big[G(\vX_h^n)-G(\R_h\vX(t_n))\big]\partial_t\vX(t_{n+1}),\ve_h^{n+1}-\ve_h^n\right)\\
			&+\left(\big[G(\R_h\vX(t_n))-G^n\big]\partial_t\vX(t_{n+1}),\ve_h^{n+1}-\ve_h^n\right):=d_{3}^1+d_{3}^2.
		\end{aligned}
	\end{equation*}
	For $d_{3}^1$, note that 
	\begin{equation*}
		\begin{aligned}
			G(\vX_h^n)-G(\R_h\vX(t_n))=&\left(|\partial_s\vX_h^n|^{2}-|\partial_s\R_h\vX(t_n)|^{2}\right)\mI\\
			&+(\alpha-1)\left(\partial_s\vX_h^n\otimes\partial_s\vX_h^n-\partial_s\R_h\vX(t_n)\otimes\partial_s\R_h\vX(t_n)\right)\\
			=&\left(|\partial_s\vX_h^n|^{2}-|\partial_s\R_h\vX(t_n)|^{2}\right)\mI\\
			&+(\alpha-1)\partial_s\vX_h^n\otimes\left(\partial_s\vX_h^n -\partial_s\R_h\vX(t_n)\right)\\
			&+(\alpha-1)\left( \partial_s\vX_h^n-\partial_s\R_h\vX(t_n)\right)\otimes\partial_s\R_hX(t_n)
		\end{aligned}
	\end{equation*}
	and the bound for $|\partial_s\vX_h^n|$, it follows
	\begin{equation}\label{d3-1}
		\begin{aligned}
			d_{3}^1&\le c\, c_2 (1+\alpha)\big\|\partial_s\ve_h^n\big\|\big\|\ve_h^{n+1}-\ve_h^n\big\|\\
			&\le \frac{c_\alpha c_1^2}{\Delta t} \epsilon\big\| \ve_h^{n+1} -\ve_h^n\big\|^2+\frac{c\,(1+\alpha)^2}{c_\alpha \epsilon}\big\|\partial_s\ve_h^{n} \big\|^2\Delta t.
		\end{aligned}
	\end{equation}
	Now let us focus on $d_{ 3}^2$
	\begin{equation*}
		\begin{aligned}
			G(\R_h\vX(t_n))-G^n=&\left(|\partial_s\R_h\vX(t_n)|^2-|\partial_s\vX(t_n)|^2\right)\mI\\
			&+(\alpha-1)\left(\partial_s\R_h\vX(t_n)\otimes\partial_s\R_h\vX(t_n)-\partial_s\vX(t_n)\otimes\partial_s\vX(t_n)\right).
		\end{aligned}
	\end{equation*}
	Substitute $\R_h\vX=\vX+\vrho_h$ then it follows
	\begin{equation*}
		\begin{aligned}
			G(\R_h\vX(t_n))-G^n&=\left(2\partial_s\vX(t_n)\cdot \partial_s\vrho_h^n+|\partial_s\vrho_h^n|^2\right)\mI\\
			&+(\alpha-1)\left(\partial_s\vrho_h^n\otimes\partial_s\vX(t_n)+\partial_s\vX(t_n)\otimes\partial_s\vrho_h^n+\partial_s\vrho_h^n\otimes\partial_s\vrho_h^n\right),
		\end{aligned}
	\end{equation*}
	which splits $d_{3}^2$ into five terms, and we denote them by $\{d_{3}^{2,i}\}_{i=1}^5$, respectively.
	
	\noindent For $d_{3}^{2,2}+d_{3}^{2,5}$ we have
	\begin{equation*}
		\begin{aligned}
			d_{3}^{2,2}+d_{3}^{2,5}&= \left(\left[|\partial_s\vrho_h^n|^2\mI +(\alpha-1) \partial_s\vrho_h^n\otimes\partial_s\vrho_h^n\right]\,
			\partial_t\vX(t_{n+1}),\ve_h^{n+1}-\ve_h^n\right)\\
			&\le (1+|\alpha-1|)\big\|\partial_s\vrho_h^n\big\|^2_{\infty}\big\| \partial_t\vX(t_{n+1})\big\| \big\| \ve_h^{n+1}-\ve_h^n\big\|,
		\end{aligned}
	\end{equation*}
	where $\|\cdot\|_{\infty}$ denotes the $L^{\infty}(I)$ norm. For $k=1$, standard interpolation estimate gives
	\begin{equation}\label{rho-est-k1}
		\big\|\partial_s\vrho_h(t)\big\|_{\infty}=\big\|\partial_s\cI_h\vX(t)-\partial_s\vX(t)\big\|_{\infty}\le ch\big\|\partial_s^{2}\vX(t)\big\|_{\infty}.
	\end{equation}
	For $k\ge 2$, utilizing the  triangle inequality and the inverse estimate it follows
	\begin{equation}\label{rho-est-k2}
		\begin{aligned}
			\big\|\partial_s\vrho_h(t)\big\|_{\infty}\le & \big\| \partial_s\Ritz_h\vX(t)-\partial_s\cI_h\vX(t)\big\|_\infty +\big\| \partial_s\ \cI_h\vX(t)-\vX(t)\big\|_\infty\\
			\le & ch^{-\half}\big\| \partial_s\Ritz_h\vX(t)-\partial_s\cI_h\vX(t)\big\|+ ch^{k}\big\|\partial_s^{k+1}\vX(t)\big\|_{\infty} \\
			\le & ch^{k-\half}\big\|\partial_s^{k+1}\vX(t)\big\|+ch^{k}\big\|\partial_s^{k+1}\vX(t)\big\|_{\infty},
		\end{aligned}
	\end{equation}
	where the last step follows from 
	\begin{equation*}
		\big\| \partial_s\Ritz_h\vX(t)-\partial_s\cI_h\vX(t)\big\|\le \big\| \partial_s\Ritz_h\vX(t)-\partial_s \vX(t)\big\|+\big\| \partial_s\cI_h\vX(t)-\partial_s \vX(t)\big\|
	\end{equation*}
	and the $H^1$ estimates of $\Ritz_h$ and $\cI_h$. Note that $2k-1\ge k+1$ for $k\ge 2$, so combining \eqref{rho-est-k1} and \eqref{rho-est-k2} we get for $k\ge 1$ 
	\begin{equation}\label{est-pho-inf}
		\big\|\partial_s\vrho_h(t)\big\|_{\infty}^2\le ch^{k+1},
	\end{equation}
	which gives
	\begin{equation*}
		\begin{aligned}
			d_{3}^{2,2}+d_{3}^{2,5}\le c(1+\alpha)h^{k+1}\big\|\ve_h^{n+1}-\ve_h^n\big\|.
		\end{aligned}
	\end{equation*}
	
	\noindent For the term $d_{3}^{2,1}$, integration by parts gives
	\begin{equation*}
		\begin{aligned}
			d_{3}^{2,1} =&-2\left(\partial_s^2\vX(t_n)\cdot \vrho_h^n\,\partial_t\vX(t_{n+1}),\ve_h^{n+1}-\ve_h^n\right)\\
			&-2\left(\partial_s\vX(t_n)\cdot \vrho_h^n\,\partial_s\partial_t\vX(t_{n+1}),\ve_h^{n+1}-\ve_h^n\right)\\
			&-2\left(\partial_s\vX(t_n)\cdot\vrho_h^n\,\partial_t\vX(t_{n+1}),\partial_s\ve_h^{n+1}-\partial_s\ve_h^n\right)\\
			\le & \,c\big\|\vrho_h^n\big\|\, \big\|\ve_h^{n+1}-\ve_h^n\big\|-2\left(\vY^{n+1},\partial_s\ve_h^{n+1}\right)+2\left(\vY^n,\partial_s\ve_h^{n}\right)
			+2\left(\vY^{n+1}-\vY^n,\partial_s\ve_h^{n+1}\right)\\
			\le & \,ch^{k+1} \big\|\ve_h^{n+1}-\ve_h^n\big\|+ch^{k+1}\Delta t \big\|\partial_s\ve_h^{n+1}\big\| -2\left(\vY^{n+1},\partial_s\ve_h^{n+1}\right)+2\left(\vY^n,\partial_s\ve_h^{n}\right),
		\end{aligned}
	\end{equation*}
	where for simplicity we denote by $\vY^n=\partial_s\vX(t_{n})\cdot\vrho_h^{n}\,\partial_t\vX(t_{n+1})$. Applying integration by parts for $d_{3}^{2,3}$ it yields
	\begin{equation*}
		\begin{aligned}
			d_{3}^{2,3}=&-(\alpha-1)\left(\vrho_h^n\otimes\partial_s^2\vX(t_n)\,\partial_t\vX(t_{n+1}),\ve_h^{n+1}-\ve_h^n\right)\\
			&-(\alpha-1)\left(\vrho_h^n\otimes\partial_s\vX(t_n)\,\partial_s\partial_t\vX(t_{n+1}),\ve_h^{n+1}-\ve_h^n\right)\\
			&-(\alpha-1)\left(\vrho_h^n\otimes\partial_s\vX(t_n)\,\partial_t\vX(t_{n+1}),\partial_s\ve_h^{n+1}-\partial_s\ve_h^n\right)\\
			\le &  \,c|1-\alpha| \big\|\vrho_h^n\big\|\, \big\|\ve_h^{n+1}-\ve_h^n\big\|-(\alpha-1)\left(\vZ_1^{n+1},\partial_s\ve_h^{n+1}\right)+(\alpha-1)\left(\vZ_1^{n},\partial_s\ve_h^{n}\right)\\
			&+(\alpha-1)\left(\vZ_1^{n+1}-\vZ_1^n,\partial_s\ve_h^{n+1}\right)\\
			\le &\,c|1-\alpha|h^{k+1} \big\|\ve_h^{n+1}-\ve_h^n\big\|+ch^{k+1}\Delta t \big\|\partial_s\ve_h^{n+1}\big\|\\
			& -(\alpha-1)\left(\vZ_1^{n+1},\partial_s\ve_h^{n+1}\right)+(\alpha-1)\left(\vZ_1^{n},\partial_s\ve_h^{n}\right),
		\end{aligned}
	\end{equation*}
	where $\vZ_1^n:=\vrho_h^n\otimes\partial_s\vX(t_n)\,\partial_t\vX(t_{n+1})$. Similarly
	\begin{equation*}
		\begin{aligned}
			d_{3}^{2,4}\le  &\,c|1-\alpha|h^{k+1} \big\|\ve_h^{n+1}-\ve_h^n\big\|+ch^{k+1}\Delta t \big\|\partial_s\ve_h^{n+1}\big\|\\
			& -(\alpha-1)\left(\vZ_2^{n+1},\partial_s\ve_h^{n+1}\right)+(\alpha-1)\left(\vZ_2^{n},\partial_s\ve_h^{n}\right),
		\end{aligned}
	\end{equation*}
	where  $\vZ_2^n:=\partial_s\vX(t_n)\otimes\vrho_h^n \,\partial_t\vX(t_{n+1})$. Sum them up we get 
	\begin{equation}\label{d3-2}
		\begin{aligned}
			d_{3}^2\le & c\,(1+\alpha) h^{k+1} \big\|\ve_h^{n+1}-\ve_h^n\big\|+ch^{k+1}\Delta t \big\|\partial_s\ve_h^{n+1}\big\|\\
			&-2\left(\vY^{n+1},\partial_s\ve_h^{n+1}\right)-(\alpha-1)\left(\vZ^{n+1},\partial_s\ve_h^{n+1}\right)\\
			&+2\left(\vY^{n},\partial_s\ve_h^{n}\right)+(\alpha-1)\left(\vZ^{n},\partial_s\ve_h^{n}\right)\\
			\le & \frac{c_\alpha c_1^2}{\Delta t}\epsilon \big\| \ve_h^{n+1} -\ve_h^n\big\|^2+c\frac{(1+\alpha)^2}{c_\alpha \epsilon}h^{2k+2}\Delta t+ch^{k+1}\Delta t \big\|\partial_s\ve_h^{n+1}\big\|\\
			&-2\left(\vY^{n+1},\partial_s\ve_h^{n+1}\right)-(\alpha-1)\left(\vZ^{n+1},\partial_s\ve_h^{n+1}\right)\\
			&+2\left(\vY^{n},\partial_s\ve_h^{n}\right)+(\alpha-1)\left(\vZ^{n},\partial_s\ve_h^{n}\right),
		\end{aligned}
	\end{equation}
	where $\vZ^n=\vZ_1^n+\vZ_2^n$. In conclusion, we have
	\begin{equation}\label{est-d}
		\begin{aligned}
			\sum_{i=1}^4 d_i\le&  5\frac{c_\alpha c_1^2}{\Delta t}\epsilon \big\| \ve_h^{n+1} -\ve_h^n\big\|^2+ch^{k+1}\Delta t \big\|\partial_s\ve_h^{n+1}\big\|+c\frac{(1+\alpha)^2}{c_\alpha\epsilon}\Delta t \big\|\partial_s\ve_h^{n} \big\|^2\\
			&-2\left(\vY^{n+1},\partial_s\ve_h^{n+1}\right)-(\alpha-1)\left(\vZ^{n+1},\partial_s\ve_h^{n+1}\right)\\
			&+2\left(\vY^{n},\partial_s\ve_h^{n}\right)+(\alpha-1)\left(\vZ^{n},\partial_s\ve_h^{n}\right)\\
			&+c\frac{(1+\alpha)^2}{c_\alpha \epsilon}(\Delta t)^3 +c\frac{(1+\alpha)^2}{c_\alpha  \epsilon}h^{2k+2}\Delta t.
		\end{aligned}
	\end{equation}
	Appealing to \eqref{est-1} and \eqref{est-left},  taking $\epsilon=\frac{1}{40}$, the first term on the right hand side of \eqref{est-d} can be absorbed and the remaining terms satisfy
	\begin{equation}\label{est-f1}
		\begin{aligned}
			&\frac{c_\alpha c_1^2}{8\Delta t} \big\| \ve_h^{n+1} -\ve_h^n\big\|^2  +\frac{1}{2}\left\|\partial_s(\ve_h^{n+1}-\ve_h^n) \right\|^2+\frac{1}{2}\left\|\partial_s\ve_h^{n+1}\right\|^2\\ &+2\left(\vY^{n+1},\partial_s\ve_h^{n+1}\right)+(\alpha-1)\left(\vZ^{n+1},\partial_s\ve_h^{n+1}\right)\\
			\le & \frac{1}{2}\left\|\partial_s\ve_h^{n}\right\|^2 +\,ch^{k+1}\Delta t \big\|\partial_s\ve_h^{n+1}\big\|+c\frac{(1+\alpha)^2}{c_\alpha}\Delta t \big\|\partial_s\ve_h^{n} \big\|^2\\
			&+2\left(\vY^{n},\partial_s\ve_h^{n}\right)+(\alpha-1)\left(\vZ^{n},\partial_s\ve_h^{n}\right)\\
			&+c\frac{(1+\alpha)^2}{c_\alpha}(\Delta t)^3 +c\frac{(1+\alpha)^2}{c_\alpha}h^{2k+2}\Delta t.
		\end{aligned}
	\end{equation}
	Utilizing $ch^{k+1}\Delta t \big\|\partial_s\ve_h^{n+1}\big\|\le \Delta t \big\|\partial_s\ve_h^{n+1}\big\|^2+c h^{2k+2}\Delta t$,  summing up the inequality for $n$ it follows
	\begin{equation}\label{est-f2}
		\begin{aligned}
			&\frac{c_\alpha c_1^2}{8\Delta t} \sum_{m=0}^n \big\| \ve_h^{m+1} -\ve_h^m\big\|^2  +\frac{1}{2}\left\|\partial_s\ve_h^{n+1}\right\|^2 +\frac{1}{2}\sum_{m=0}^{n}\left\|\partial_s(\ve_h^{m+1}-\ve_h^m) \right\|^2\\ &+2\left(\vY^{n+1},\partial_s\ve_h^{n+1}\right)+(\alpha-1)\left(\vZ^{n+1},\partial_s\ve_h^{n+1}\right)\\
			\le & \Delta t \big\|\partial_s\ve_h^{n+1}\big\|^2+ c\frac{(1+\alpha)^2}{c_\alpha} \Delta t\sum_{m=1}^n \big\|\partial_s\ve_h^{m} \big\|^2+\tilde{c}\frac{(1+\alpha)^2}{c_\alpha}\big((\Delta t)^2+ h^{2k+2}\big),
		\end{aligned}
	\end{equation}
	where we utilized $\ve_h^0=\vec{0}$. Note that 
	\begin{equation*}
		2\left|\left(\vY^{n+1},\partial_s\ve_h^{n+1}\right) \right|\le ch^{k+1}\left\|\partial_s\ve_h^{n+1}\right\| \le \frac{1}{8}\left\|\partial_s\ve_h^{n+1}\right\|^2+ch^{2k+2}
	\end{equation*}
	and
	\begin{equation*}
		\left|(\alpha-1)\left(\vZ^{n+1},\partial_s\ve_h^{n+1}\right) \right| \le c|\alpha-1| h^{k+1}\left\|\partial_s\ve_h^{n+1}\right\| \le \frac{1}{8}\left\|\partial_s\ve_h^{n+1}\right\|^2+c(1+\alpha)^2h^{2k+2},
	\end{equation*}
	then we get from \eqref{est-f2}
	\begin{equation}\label{est-f3}
		\begin{aligned}
			&\frac{c_\alpha c_1^2}{8\Delta t} \sum_{m=0}^n \big\| \ve_h^{m+1} -\ve_h^m\big\|^2  +\frac{1}{4}\left\|\partial_s\ve_h^{n+1}\right\|^2 +\frac{1}{2}\sum_{m=0}^{n}\left\|\partial_s(\ve_h^{m+1}-\ve_h^m) \right\|^2\\
			\le & \Delta t \big\|\partial_s\ve_h^{n+1}\big\|^2+ c\frac{(1+\alpha)^2}{c_\alpha} \Delta t\sum_{m=1}^n \big\|\partial_s\ve_h^{m} \big\|^2+\tilde{c}\frac{(1+\alpha)^2}{c_\alpha}\big((\Delta t)^2+ h^{2k+2} \big).
		\end{aligned}
	\end{equation}
	Let $\Delta t\le \frac{1}{8}$ then \eqref{est-f3} implies
	\begin{equation}\label{est-f4}
		\begin{aligned}
			&\frac{c_\alpha c_1^2}{\Delta t} \sum_{m=0}^n \big\| \ve_h^{m+1} -\ve_h^m\big\|^2 +\left\|\partial_s\ve_h^{n+1}\right\|^2 \\
			\le &  c\frac{(1+\alpha)^2}{c_\alpha} \Delta t\sum_{m=1}^n \big\|\partial_s\ve_h^{m} \big\|^2+\tilde{c}\frac{(1+\alpha)^2}{c_\alpha}\big((\Delta t)^2+ h^{2k+2} \big).
		\end{aligned}
	\end{equation}
	Discrete Gronwall's inequality gives
	\begin{equation}\label{est-en-H1}
		\left\|\partial_s\ve_h^{n+1}\right\|^2 \le \tilde{c}\frac{(1+\alpha)^2}{c_\alpha}\big((\Delta t)^2+ h^{2k+2} \big)e^{\frac{c(1+\alpha)^2}{c_\alpha}t_{n}}.
	\end{equation}
	So by choosing $K_\alpha^2= \tilde{c}\frac{(1+\alpha)^2}{c_\alpha}$ and $\beta_\alpha=c\frac{(1+\alpha)^2}{c_\alpha}$ the desired estimate of $\partial_s\ve_h^m$ for $m=n+1$ follows. 
	
	Next we turn to the $L^2$-error estimate. Substitute the bound $\|\partial_s\ve_h^m\|^2$  into \eqref{est-f4} then we get
	\begin{equation*}
		\frac{1}{\Delta t} \sum_{m=0}^n \big\| \ve_h^{m+1} -\ve_h^m\big\|^2\le c(\alpha) \big((\Delta t)^2+h^{2k+2}\big).
	\end{equation*}
	Suppose $\|\ve_h^{l}\|=\max_{n}\|\ve_h^n\|$, then
	\begin{equation*}
		\begin{aligned}
			\big\| \ve_h^{l}\big\|^2&=\sum_{n=0}^{l-1}\big(\|\ve_h^{n+1}\|^2- \|\ve_h^n\|^2 \big)\\
			&=\sum_{n=0}^{l-1}\big(\|\ve_h^{n+1}\|- \|\ve_h^n\| \big)\big(\|\ve_h^{n+1}\|+\|\ve_h^n\| \big)\\
			&\le 2 \|\ve_h^l\|\sum_{n=0}^{l-1}\|\ve_h^{n+1}-\ve_h^n\|\\
			&\le  2\|\ve_h^l\|\sum_{n=0}^{l-1}\sqrt{\Delta t}\,\frac{\|\ve_h^{n+1}-\ve_h^n\|}{\sqrt{\Delta t}}\\
			&\le 2\sqrt{T}\|\ve_h^l\| \left(\sum_{n=0}^{l-1}\,\frac{\|\ve_h^{n+1}-\ve_h^n\|^2}{\Delta t}\right)^{1/2},
		\end{aligned}
	\end{equation*}
	which gives
	\begin{equation}\label{est-L2}
		\big\| \ve_h^{l}\big\| \le 2\sqrt{T} \left(\sum_{n=0}^{l-1}\,\frac{\|\ve_h^{n+1}-\ve_h^n\|^2}{\Delta t}\right)^{1/2}\le c(\alpha) \big(\Delta t+h^{k+1}\big).
	\end{equation}
	Thus the desired estimate for $\ve_h^m$ follows.
\end{proof}

\begin{remark}\label{remark-e0}
	In fact, if we set $n=0$ in \eqref{est-f1} with noting that $\ve_h^0=\vec{0}$ then it is easy to get 
	\begin{equation*}
		\frac{1}{\Delta  t}\|\ve_h^1\|^2+ \|\partial_s\ve_h^1\|^2\le \sigma_\alpha (h^{2k+2}+(\Delta t)^3),
	\end{equation*}
	where $\sigma_\alpha$ is a constant independent of $\Delta t$ and $h$ but dependent on $\alpha$.
\end{remark}

\section{Crank-Nicolson scheme}\label{sec-CN}
In this section we consider the following  linearized Crank-Nicolson scheme: Find $\vX_h^{n+1}\in S(I)^2$ with  $\vX_h^0=\R_h\vX(0)$ such that
\begin{equation}\label{CN}
	\left(G(\hX_h^{n+\half})\delta_t\vX_h^{n+1},\vxi_h\right)+\left(\partial_s\vX_h^{n+\half},\partial_s\vxi_h\right)=0\quad \forall \vxi_h\in S(I)^2,
\end{equation}
where $\vX_h^{n+\half}:=\frac{1}{2}(\vX_h^{n+1}+\vX_h^n)$ for $n\ge 0$ and $\hX_h^{n+\half}=\frac{3}{2}\vX_h^n-\frac{1}{2}\vX_h^{n-1}$ for $n\ge 1$. 
For $n=0$, instead of the extrapolation we set $\hX_h^{\half}$ to be the solution obtained by our Euler scheme with a time step of $\frac{1}{2}\Delta t$. Like before, to present the error estimate we denote  
\begin{equation*}
	\ve_h^n=\R_h\vX(t_n)-\vX_h^n\quad \mbox{and}\quad 	\ve_h^{n+\half}=\frac{\ve_h^{n+1}+\ve_h^n}{2}\quad\mbox{for }n\ge 0.
\end{equation*}
\begin{theorem}\label{est-CN-H1}
	Suppose \eqref{Elliott-Frit-1} has a unique solution $\vX(t,s)$ which is smooth enough with respect to both $t$ and $s$ up to  some $t> T$. Then there exists  some $h_0$ such that for $h\le h_0$ and $t\le h^{1/3}$ being small enough, our Crank-Nicolson scheme \eqref{CN} has a unique solution and satisfies
	\begin{equation}\label{est-CN}
		\big\|\ve_h^m \big\| + \big\|\partial_s\ve_h^{m}\big\|\le c(\alpha)\left((\Delta t)^2+h^{k+1}\right) \quad \mbox{for }m\Delta t\le T,
	\end{equation}
	where $c(\alpha)$ is a constant that is independent of $\Delta t$ and $h$ but depends on $\alpha$. Together with the estimate for $\vrho_h$ we have the optimal error bound
	\begin{equation}\label{est-CN-total}
		\big\|\partial_s^p(\vX_h^m-\vX(t_m)) \big\| \le  c(\alpha)\left((\Delta t)^2 +h^{k+1-p}\right),\quad p=0,1.
	\end{equation}
\end{theorem}
\begin{proof}
	The existence and uniqueness follows directly as before by noting at each time level the numerical scheme is an elliptic system. Next we sketch the error estimate. The error equation for our Crank-Nicolson scheme is 
	\begin{equation}\label{eqn-err}
		\left(G(\hX_h^{n+\half})\delta_t \ve_h^{n+1},\vxi_h\right)+\left(\partial_s\ve_h^{n+\half},\partial_s\vxi_h\right)=(d_{n+1},\vxi_h),
	\end{equation}
	where 
	\begin{equation*}
		\begin{aligned}
			d_{n+1}= & G(\hX_h^{n+\half})\delta_t\R_h\vX(t_{n+1})-G^{n+\half}\partial_t\vX(t_{n+\half})\\
			& +\frac{\partial_s\R_h\vX(t_{n+1}) +\partial_s\R_h\vX(t_n)}{2}-\partial_s\R_h\vX(t_{n+\half})
		\end{aligned}
	\end{equation*}
	and $G^{n+\half}:=G(\vX(t_{n+\half}))$. Like before, we will still first prove by mathematical induction that there exist some $K_\alpha$ and $\beta_\alpha$ such that
	\begin{equation}\label{est-CN-R}
		\left\|\partial_s\ve_h^{m}\right\|^2\le K^2_\alpha e^{\beta_\alpha t_{m-1}}((\Delta t)^4+h^{2k+2})\quad \mbox{for }m\ge 1.
	\end{equation}
	
	We leave the case of $m=1$ at the end since the corresponding procedures are similar, though the prediction $\hX_h^{\half}$ is from our Euler method rather than the linear extrapolation.  Suppose $m\le n$ is verified then by the triangle inequality and the inverse estimate we can get
	\begin{equation*}
		\frac{c_1}{2}\le |\partial_s\hX_h^{n+\half}|\le 2c_2
	\end{equation*}
	when $h^{k+\half}$ and $h^{-\half}(\Delta t)^2$ are small enough. 
	
	Denote $R_h(t)=\R_h\vX(t)$ and for notational simplicity, we use $R_h^{n+\theta}$ to represent $R_h(t_{n+\theta})$ for $n\ge 0$ and $\theta\in(0,1)$. Furthermore, we let $\hRh^{n+\half}=\frac{3}{2}R_h^{n}-\frac{1}{2}R_h^{n-1}(n\ge 1)$ be the linear extrapolation. Take $\vxi_h=\ve_h^{n+1}-\ve_h^n$ then the  LHS yields
	\begin{equation}\label{left-side}
		\begin{aligned}
			\mathrm{LHS}\ge\frac{c_\alpha c_1^2}{4\Delta t}\big\|\ve_h^{n+1}-\ve_h^n\big\|^2+\frac{1}{2}\big\|\partial_s \ve_h^{n+1}\big\|^2-\frac{1}{2}\big\|\partial_s\ve_h^{n}\big\|^2
		\end{aligned}
	\end{equation}
	and the defect term is
	\begin{equation*}
		\begin{aligned}
			\left(d_{n+1},\ve_h^{n+1}-\ve_h^n\right)&=\left(G(\hX_h^{n+\half})\delta_t\vrho_h^{n+1},\ve_h^{n+1}-\ve_h^n\right)\\
			&+\left( G(\hX_h^{n+\half})\big(\delta_t\vX(t_{n+1})-\partial_t\vX(t_{n+\half})\big),\ve_h^{n+1}-\ve_h^n\right)\\
			&+\left(\big[G(\hX_h^{n+\half})-G(\hX^{n+\half})\big]\partial_t\vX(t_{n+\half}),\ve_h^{n+1}-\ve_h^n\right)\\
			&+\left(\big[G(\hX^{n+\half})-G^{n+\half}\big]\partial_t\vX(t_{n+\half}),\ve_h^{n+1}-\ve_h^n\right)\\
			&+\left(\frac{1}{2}\partial_s (R_h^{n+1}+R_h^n)-\partial_s R_h^{n+\half},\partial_s\ve_h^{n+1}-\partial_s \ve_h^n\right)\\
			:&=d_{1}+d_{2}+d_{3}+d_4+d_5,
		\end{aligned}
	\end{equation*}
	where $\hX^{n+\half}=\frac{3}{2}\vX(t_n)-\frac{1}{2}\vX(t_{n-1})$. Similar to the Eulerian procedures we can get 
	\begin{equation}\label{d1-CN}
		d_1 \le \frac{c_\alpha c_1^2}{\Delta t}\epsilon \big\| \ve_h^{n+1} -\ve_h^n\big\|^2+\frac{c(1+\alpha)^2 h^{2k+2}}{c_\alpha\epsilon}\Delta t,
	\end{equation}
	\begin{equation}\label{d2-CN}
		d_{2}\le  \frac{c_\alpha c_1^2}{\Delta t} \epsilon\big\| \ve_h^{n+1} -\ve_h^n\big\|^2+\frac{c\,(1+\alpha)^2 }{c_\alpha\epsilon}(\Delta t)^5,
	\end{equation}
	and
	\begin{equation}\label{d4-CN}
		d_{4}\le c (1+\alpha)(\Delta t)^2\big\|\ve_h^{n+1}-\ve_h^n\big\|\le \frac{c_\alpha c_1^2}{\Delta t} \epsilon\big\| \ve_h^{n+1} -\ve_h^n\big\|^2+\frac{c\,(1+\alpha)^2}{c_\alpha\epsilon}(\Delta t)^5.
	\end{equation}
	For $d_5$, Taylor expansion gives
	\begin{equation*}
		\begin{aligned}
			\frac{1}{2}(\partial_s R_h^{n+1}+\partial_s R_h^n)-\partial_s R_h^{n+\half}&=\frac{(\Delta t)^2}{16}\partial_s \left(\partial_t^{2}\R_h\vX(t_{n+\theta_1})+\partial_t^{2}\R_h\vX(t_{n+\theta_2})\right)\\
			&=\frac{(\Delta t)^2}{16}\partial_s \left(\R_h\partial_t^{2}\vX(t_{n+\theta_1})+\R_h\partial_t^{2}\vX(t_{n+\theta_2})\right),
		\end{aligned}
	\end{equation*}
	where $\theta_1,\theta_2\in(0,1)$, and we utilized that $\partial_t$ commutes with $\R_h$. So appealing to the definition of $\R_h$, using integration by parts  we get
	\begin{equation}\label{d5-CN}
		\begin{aligned}
			d_5=&\frac{(\Delta t)^2}{16} \left(\partial_s\R_h\partial_t^{2}\vX(t_{n+\theta_1})+\partial_s\R_h\partial_t^{2}\vX(t_{n+\theta_2}),\partial_s\ve_h^{n+1}-\partial_s \ve_h^n \right)\\
			= & \frac{(\Delta t)^2}{16} \left(\partial_s \partial_t^{2}\vX(t_{n+\theta_1})+\partial_s \partial_t^{2}\vX(t_{n+\theta_2}),\partial_s\ve_h^{n+1}-\partial_s \ve_h^n \right)\\
			=&\frac{(\Delta t)^2}{16} \left(-\partial_s^2 \partial_t^{2}\vX(t_{n+\theta_1})-\partial_s^2 \partial_t^{2}\vX(t_{n+\theta_2}),\ve_h^{n+1}-\ve_h^n \right)\\
			\le & c (\Delta t)^2 \big\|\ve_h^{n+1}-\ve_h^n \big\| \\
			\le &  \frac{c_\alpha c_1^2}{\Delta t}\epsilon \big\| \ve_h^{n+1} -\ve_h^n\big\|^2+\frac{c}{c_\alpha\epsilon}(\Delta t)^5.
		\end{aligned}
	\end{equation}
	Now  we split $d_3$ it into 
	\begin{equation}\label{d3}
		\begin{aligned}
			d_3=&\left(\big[G(\hX_h^{n+\half})-G(\hRh^{n+\half})\big]\partial_t\vX(t_{n+\half}),\ve_h^{n+1}-\ve_h^n\right)\\
			&+\left(\big[ G(\hRh^{n+\half})-G(\hX^{n+\half})\big]\partial_t\vX(t_{n+\half}),\ve_h^{n+1}-\ve_h^n\right)\\
			:&=d_{3}^1+d_{3}^2.
		\end{aligned}
	\end{equation}
	Repeat the same arguments as for $d_3^1$ in the Eulerian case then we get
	\begin{equation}\label{d31}
		\begin{aligned}
			d_{3}^1\le &c (1+\alpha)\big\|\partial_s(\hX_h^{n+\half}-\hRh^{n+\half}) \big\| \big\|\ve_h^{n+1}-\ve_h^n\big\|\\
			\le & \frac{c_\alpha c_1^2}{\Delta t} \epsilon\big\| \ve_h^{n+1} -\ve_h^n\big\|^2+\frac{c\,(1+\alpha)^2}{c_\alpha \epsilon}\left(\big\|\partial_s\ve_h^{n} \big\|^2+\big\|\partial_s\ve_h^{n-1} \big\|^2\right)\Delta t.
		\end{aligned}
	\end{equation}
	For $d_3^2$ we substitute $\hRh^{n+\half}=\hX^{n+\half}+\hrho_h^{n+\half}$ then
	\begin{equation*}
		\begin{aligned}
			& G(\hRh^{n+\half})-G(\hX^{n+\half})\\
			=&\left(2\partial_s\hX^{n+\half}\cdot \partial_s\hrho_h^{n+\half}+|\partial_s\hrho_h^{n+\half}|^2\right)\mI\\
			&+(\alpha-1)\left(\partial_s\hrho_h^{n+\half}\otimes\partial_s\hX^{n+\half}+\partial_s\hX^{n+\half}\otimes\partial_s\hrho_h^{n+\half}+\partial_s\hrho_h^{n+\half}\otimes\partial_s\hrho_h^{n+\half}\right),
		\end{aligned}
	\end{equation*}
	which, as the Eulerian case, splits $d_{3}^2$ into five terms $\{d_{3}^{2,i}\}_{i=1}^5$. Note that $\hrho_h^{n+\half}=\frac{3}{2}\vrho_h^n-\frac{1}{2}\vrho_h^{n-1}$, so for $n\ge 1$ by defining 
	\begin{equation*}
		\vmY^n=\big(\partial_s\hX^{n+\half}\cdot \hrho_h^{n+\half}\big) \partial_t\vX(t_{n+\half})
	\end{equation*}
	and
	\begin{equation*}
		\vmZ^n=\big[\hrho_h^{n+\half}\otimes\partial_s\hX^{n+\half}+\partial_s\hX^{n+\half}\otimes\hrho_h^{n+\half}\big]\partial_t\vX(t_{n+\half}),
	\end{equation*}
	one can get similar estimate as before
	\begin{equation}\label{d32}
		\begin{aligned}
			d_{3}^2\le & \frac{c_\alpha c_1^2}{\Delta t}\epsilon \big\| \ve_h^{n+1} -\ve_h^n\big\|^2+c\frac{(1+\alpha)^2}{c_\alpha \epsilon}h^{2k+2}\Delta t+ch^{k+1}\Delta t \big\|\partial_s\ve_h^{n+1}\big\|\\
			&-2\left(\vmY^{n+1},\partial_s\ve_h^{n+1}\right)-(\alpha-1)\left(\vmZ^{n+1},\partial_s\ve_h^{n+1}\right)\\
			&+2\left(\vmY^{n},\partial_s\ve_h^{n}\right)+(\alpha-1)\left(\vmZ^{n},\partial_s\ve_h^{n}\right).
		\end{aligned}
	\end{equation}
	Sum the defects up and take $\epsilon=\frac{1}{48}$, then we get the bound
	\begin{equation}\label{est-f1-CN}
		\begin{aligned}
			&\frac{c_\alpha c_1^2}{8\Delta t} \big\| \ve_h^{n+1} -\ve_h^n\big\|^2  +\frac{1}{2}\left\|\partial_s\ve_h^{n+1}\right\|^2 +2\left(\vmY^{n+1},\partial_s\ve_h^{n+1}\right)+(\alpha-1)\left(\vmZ^{n+1},\partial_s\ve_h^{n+1}\right)\\
			\le & \frac{1}{2}\left\|\partial_s\ve_h^{n}\right\|^2 +\,ch^{k+1}\Delta t \big\|\partial_s\ve_h^{n+1}\big\|+c\frac{(1+\alpha)^2}{c_\alpha}\Delta t \left(\big\|\partial_s\ve_h^{n} \big\|^2+\big\|\partial_s\ve_h^{n-1} \big\|^2\right)\\
			&+2\left(\vmY^{n},\partial_s\ve_h^{n}\right)+(\alpha-1)\left(\vmZ^{n},\partial_s\ve_h^{n}\right) +c\frac{(1+\alpha)^2}{c_\alpha}(\Delta t)^5 +c\frac{(1+\alpha)^2}{c_\alpha}h^{2k+2}\Delta t.
		\end{aligned}
	\end{equation}
	Summing up the inequality for $n\ge 1$ it follows
	\begin{equation}\label{est-f2-CN}
		\begin{aligned}
			&\frac{c_\alpha c_1^2}{8\Delta t} \sum_{m=1}^n \big\| \ve_h^{m+1} -\ve_h^m\big\|^2 + \frac{1}{2}\left\|\partial_s\ve_h^{n+1}\right\|^2  +2\left(\vmY^{n+1},\partial_s\ve_h^{n+1}\right)+(\alpha-1)\left(\vmZ^{n+1},\partial_s\ve_h^{n+1}\right)\\
			\le & \Delta t \big\|\partial_s\ve_h^{n+1}\big\|^2+ c\frac{(1+\alpha)^2}{c_\alpha} \Delta t\sum_{m=1}^n \big\|\partial_s\ve_h^{m} \big\|^2+c\frac{(1+\alpha)^2}{c_\alpha}\big((\Delta t)^4+ h^{2k+2}\big)\\
			&+\frac{1}{2}\big\|\partial_s \ve_h^1\big\|^2 +2\left(\vmY^{1},\partial_s\ve_h^{1}\right)+(\alpha-1)\left(\vmZ^{1},\partial_s\ve_h^{1}\right).
		\end{aligned}
	\end{equation}
	Appealing to the definition of $\vmY^m$ and $\vmZ^m$ then we can get
	\begin{equation*}
		2\left(\vmY^{m},\partial_s\ve_h^{m}\right)+(\alpha-1)\left(\vmZ^{m},\partial_s\ve_h^{m}\right)\le c\,(1+\alpha) h^{k+1}\big\| \ve_h^m\big\|\le \epsilon\big\| \ve_h^m\big\|^2+\frac{c(1+\alpha)^2}{\epsilon}h^{2k+2}.
	\end{equation*}
	Take $\epsilon=\frac{1}{4}$ for $m=n+1$ and $\epsilon=\frac{1}{2}$ for $m=1$ then it follows
	\begin{equation}\label{est-f3-CN}
		\begin{aligned}
			&\frac{c_\alpha c_1^2}{8\Delta t} \sum_{m=1}^n \big\| \ve_h^{m+1} -\ve_h^m\big\|^2 + \big(\frac{1}{4}-\Delta t\big) \left\|\partial_s\ve_h^{n+1}\right\|^2 \\
			\le & c\frac{(1+\alpha)^2}{c_\alpha} \Delta t\sum_{m=1}^n \big\|\partial_s\ve_h^{m} \big\|^2+\big\|\partial_s \ve_h^1\big\|^2 +\tilde{c}\frac{(1+\alpha)^2}{c_\alpha}\big((\Delta t)^4+ h^{2k+2}\big).
		\end{aligned}
	\end{equation}
	By further assuming $\Delta t\le \frac{1}{8}$ and utilizing the following estimate that we will prove 
	\begin{equation}\label{est-e1}
		\|\partial_s \ve_h^1\|^2\le \sigma'_\alpha\big((\Delta t)^4+ h^{2k+2} \big),
	\end{equation}
	where $\sigma_\alpha'$ is a constant independent of $h$ and $\Delta t$ but dependent on $\alpha$,  together with the discrete Gronwall's inequality we get
	\begin{equation}\label{est-en-H1-CN}
		\left\|\partial_s\ve_h^{n+1}\right\|^2 \le \left(\sigma'_\alpha + \tilde{c}\frac{(1+\alpha)^2}{c_\alpha}\right)e^{\frac{c(1+\alpha)^2}{c_\alpha}t_{n}}\big((\Delta t)^4+ h^{2k+2} \big).
	\end{equation}
	So choosing $K^2_\alpha= \sigma'_\alpha+\tilde{c}\frac{(1+\alpha)^2}{c_\alpha}$ and $\beta_\alpha=\frac{c(1+\alpha)^2}{c_\alpha}$ the desired estimate of $\|\partial_s\ve_h^m\|$ for $m=n+1$ follows. 
	
	Now let us turn to \eqref{est-e1}. Note that
	\begin{equation*}
		\begin{aligned}
			\left(d_{0},\ve_h^1-\ve_h^0\right)&=\left(G(\hX_h^{\half})\delta_t\vrho_h^{1},\ve_h^{1}-\ve_h^0\right)\\
			&+\left( G(\hX_h^{\half})\big(\delta_t\vX(t_{1})-\partial_t\vX(t_{\half})\big),\ve_h^{1}-\ve_h^0\right)\\
			&+\left(\big[G(\hX_h^{\half})-G(R_h^{\half})\big]\partial_t\vX(t_{\half}),\ve_h^{1}-\ve_h^0\right)\\
			&+\left(\big[G(R_h^{\half})-G^{\half}\big]\partial_t\vX(t_{\half}),\ve_h^{1}-\ve_h^0\right)\\
			&+\left(\frac{1}{2}\partial_s (R_h^{1}+R_h^0)-\partial_s R_h^{\half},\partial_s\ve_h^{1}-\partial_s \ve_h^0\right)\\
			:&=D_{1}+D_{2}+D_{3}+D_4+D_5.
		\end{aligned}
	\end{equation*}
	For $D_1$, $D_2$ and $D_5$, they yield the similar estimates as $d_1$, $d_2$ and $d_5$
	\begin{equation}\label{D1-CN}
		D_1 \le \frac{c_\alpha c_1^2}{\Delta t}\epsilon \big\| \ve_h^{1} -\ve_h^0\big\|^2+\frac{c(1+\alpha)^2 h^{2k+2}}{c_\alpha\epsilon}\Delta t,
	\end{equation}
	\begin{equation}\label{D2-CN}
		D_{2}\le  \frac{c_\alpha c_1^2}{\Delta t} \epsilon\big\| \ve_h^{1} -\ve_h^0\big\|^2+\frac{c\,(1+\alpha)^2 }{c_\alpha\epsilon}(\Delta t)^5,
	\end{equation}
	and 
	\begin{equation}\label{D5-CN}
		D_{5}\le  \frac{c_\alpha c_1^2}{\Delta t} \epsilon\big\| \ve_h^{1} -\ve_h^0\big\|^2+\frac{c}{c_\alpha\epsilon}(\Delta t)^5.
	\end{equation}
	Note that $\hX_h^{\half}$ is the solution of our Euler scheme with a step size of $\frac{\Delta t}{2}$, it satisfies the error bound $\big\|\partial_s(R_h^{\half}-\hX_h^{\half})\big\|^2\le \sigma_\alpha (h^{2k+2}+(\Delta t/2)^{3})$ in \Cref{remark-e0}. Repeat similar procedures as we did for $d_3^1$ then it follows
	\begin{equation}\label{D3-CN}
		\begin{aligned}
			D_3&\le c(1+\alpha)\big\| \partial_s(R_h^{\half}-\hX_h^{\half})\big\| \big\|\ve_h^{1}-\ve_h^0\big\|\\
			&\le \frac{c_\alpha c_1^2}{\Delta t} \epsilon\big\| \ve_h^{1} -\ve_h^0\big\|^2+ \frac{c(1+\alpha)^2}{c_\alpha\epsilon}\big\| \partial_s(R_h^{\half}-\hX_h^{\half})\big\|^2 \Delta t\\
			&\le  \frac{c_\alpha c_1^2}{\Delta t} \epsilon\big\| \ve_h^{1} -\ve_h^0\big\|^2+\frac{c\, \sigma_\alpha(1+\alpha)^2 }{\epsilon}\left(\frac{1}{8}(\Delta t)^4+h^{2k+2}\Delta t\right).
		\end{aligned}
	\end{equation}
	For $D_4$, substitute $R_h^{\half}=\vX(t_{\half})+\vrho_h^{\half}$ and integrate by parts to transfer the spatial derivative before $\vrho_h^{\half}$ to other terms then we get the bound
	\begin{equation}\label{D4-CN}
		\begin{aligned}
			D_4	\le & c(1+\alpha) h^{k+1}\left(\big\| \ve_h^{1} -\ve_h^0\big\|+\big\|\partial_s( \ve_h^{1} -\ve_h^0)\big\|\right)\\
			\le & \frac{c_\alpha c_1^2}{\Delta t}\epsilon \big\| \ve_h^{1} -\ve_h^0\big\|^2+\frac{1}{4}\big\|\partial_s(\ve_h^{1}-\ve_h^0)\big\|^2 +c\frac{(1+\alpha)^2}{c_\alpha \epsilon}h^{2k+2}\Delta t+c(1+\alpha)^2h^{2k+2}.
		\end{aligned}
	\end{equation}
	
	The lower bound for LHS given in \eqref{left-side} still holds provided that $\frac{c_1}{2}\le |\partial_s \hX_h^{\half}|\le 2c_2$, which further requires $ch^{-\half}\|\partial_s\ve_h^{\half}\|\le \frac{c_1}{2}$ or equivalently, both $h$ and $h^{-1}(\Delta t)^3$ being small enough by virtue of \Cref{remark-e0}. So choosing $\epsilon=\frac{1}{40}$ together with  \eqref{left-side} and $\ve_h^0=\vec{0}$, the desired estimate \eqref{est-e1} follows.
	
	Repeat the same arguments as we did for $\|\ve_h^m\|$ in the Eulerian case then we complete the proof. 
\end{proof}

\section{Numerical examples}\label{sec-exp}

In this section we provide some numerical examples to test the convergence in both space and time. We also present some numerical tests to demonstrate  the advantage on mesh distribution. Throughout, we always use $N$ to denote the number of subintervals of $I$. All the computations are implemented in \texttt{Python} with the FEM package {\texttt{Firedrake}} \cite{Rathgeber2017Firedrake}.

\begin{example}\label{con_test_csf}\upshape 
	We use the standard example to test the convergence of our Euler scheme and Crank-Nicolson scheme. We simply set $\alpha=0.5$ in this example. The initial curve $\Gamma(0)$ is a circle with radius $r=1$ parameterized by $\vX^0(s)=(\cos s,\sin s)$ with $s\in [0,2\pi)$. Under such initial setting one can check $\vX(t,s)=\sqrt{1-2t}(\cos s,\sin s)$ is a solution of the curve shortening-DeTurck flow.  In this example the partition of $I$ is uniform. To test the convergence in space we use our Crank-Nicolson scheme in time with a uniform stepsize of $\Delta t=\frac{T}{M}$, where  $T=0.48$ and $M=3\times 10^5$. We  compute the  $L^2$ and $H^1$ error by $\|\vX_h^M-\vX(T)\|$ and $\|\partial_s(\vX_h^M-\vX(T))\|$, respectively.  \Cref{CT-csf-s} presents errors under different  $h$ and $k$. The convergence rates are then obtained by standard  procedures. One can observe the results consist with our theoretical analysis.
	\begin{table}[htbp]
		\centering
		\caption{\Cref{con_test_csf}: $L^2$ and $H^1$ errors under different $h$. }\label{CT-csf-s}
		\begin{tabular}{ccccccc}
			\toprule
			& $h$             & $\frac{2\pi}{8}$        & $\frac{2\pi}{16}$     & $\frac{2\pi}{32}$   & $\frac{2\pi}{64}$  & $\frac{2\pi}{128}$ \\ 
			\bottomrule
			\multirow{4}{*}{$k=1$} & $L^2$ Error   & 2.190e-01 &4.143e-02  &9.846e-03 & 2.432e-03 & 6.059e-04\\ 
			& Conv. rate     & --            &  2.40     & 2.07      & 2.02     & 2.01   \\
			\cmidrule{3-7}
			&$H^1$ Error   		 & 2.284e-01 &6.670e-02	 &2.957e-02 &1.435e-02	& 7.122e-03  \\ 
			& Conv. rate         & --       &  1.78 & 1.18    & 1.04       & 	1.01 \\			
			\cmidrule{1-7} 
			\multirow{4}{*}{$k=2$}& $L^2$ Error    &    1.307e-03 & 1.713e-04  & 2.172e-05  & 2.725e-06  &  3.418e-07  \\ 
			&  Conv. rate    & --           &  2.93      &2.98        &	3.00     &  3.00    \\
			\cmidrule{2-7}
			& $H^1$ Error         & 1.164e-02   & 2.889e-03  & 7.208e-04  & 1.801e-04 &  4.502e-05  \\ 
			& Conv. rate          & --          & 2.00       & 2.00       &2.00       & 2.00      \\
			\cmidrule{1-7} 
			\multirow{4}{*}{$k=3$} & $L^2$ Error        &   5.414e-05 &  3.225e-06  & 1.991e-07 &  1.242e-08  & 8.652e-10 \\ 
			& Conv. rate      & --          & 4.07        & 4.02      &   4.00      & 3.84    \\
			\cmidrule{2-7}
			& $H^1$ Error     & 7.611e-04   & 9.549e-05   & 1.195e-05 &  1.494e-06 &  1.868e-07\\ 
			& Conve. rate       & --          & 3.00        & 3.00      &  3.00      & 3.00      \\
			\bottomrule
		\end{tabular}
	\end{table}	
	
	For the convergence in time, we fix the spatial mesh with $h=\frac{2\pi}{320}$ and use finite elements of degree $k=3$ to reduce the pollution from spatial errors.  \Cref{CT-csf-t} presents  the $H^1$ errors under various stepsizes, from which one can observe that our Euler scheme is of first order accuracy and our Crank-Nicolson scheme is of second order accuracy. The $L^2$ error is very close to the $H^1$ error so we omit it here.
	\begin{table}[htbp]
		\centering
		\caption{\Cref{con_test_csf}: $H^1$ errors under different $\Delta t$, $\Delta t_0=1.2\times 10^{-2}$.}\label{CT-csf-t}
		\begin{tabular}{cccccc}
			\toprule
			$\Delta t$             & $\Delta t_0$  & $2^{-1}\Delta t_0$     & $2^{-2}\Delta t_0$   & $2^{-3}\Delta t_0$  & $2^{-4}\Delta t_0$ \\ 
			\cmidrule{1-6}
			Euler 	       &  2.211e-01    & 1.334e-01              & 7.589e-02            & 4.113e-02           & 2.154e-02 \\ 
			Conv. rate  & --           &0.72              &   0.81           &    0.88        &       0.93     \\
			\cmidrule{1-6}		
			Crank-Nicolson   & 1.550e-02 &  4.634e-03  & 1.261e-03 &  3.275e-04 &  8.329e-05\\ 
			Conv. rate  & --            &1.74          &  1.88            & 1.95              & 1.98    \\ 
			\bottomrule
		\end{tabular}
	\end{table}	
\end{example}

\begin{example}\label{test-2}\upshape 
	In this example, we present some numerical results of our Crank-Nicolson scheme with finite elements of degree $k=4$. In all the tests of this example, we  set $\alpha=0.2$ and $N=50$. The following three initial curves  are considered with $s\in[0,2\pi)$:
	\begin{equation*}
		\begin{aligned}
			a)&  \qquad \vX_0=\left[
			\begin{matrix}
				\cos {2s}\cos{s}\\
				\cos {2s}\sin{s}
			\end{matrix}
			\right],\\
			b) & \qquad \vX_0=\left[
			\begin{matrix}
				(1+0.65\sin 7{s})\cos{s}\\
				(1+0.65\sin 7{s})\sin{s}
			\end{matrix}
			\right],\\
			c) &\qquad \vX_0=\left[
			\begin{matrix}
				\cos{s}\\
				0.5\sin{s}+\sin(\cos{s})+\sin{s}(0.2+\sin{s}\sin^2{ 3s})
			\end{matrix}
			\right].
		\end{aligned}
	\end{equation*}
	The numerical curves for cases a) to c) are presented in \Cref{fig-ex1,fig-ex2,fig-ex3}, respectively, where we use {\it black dots} to mark the positions of $\vX_h^m(s_j)$ with $s_j=\frac{2j\pi}{N}$, and the gray lines between them are polynomials of degree $4$.  As we expected, since  $I$ is uniformly divided, the vertices in all the tests approach equidistributed due to the effect of the harmonic map heat flow acting on the tangent space. One also can observe that $50$ subintervals are enough to capture the details of the evolution. 
	\begin{figure}
		\centering
		\subfigure[Initial curve of a)]{\includegraphics[width=.31\textwidth,trim=120 270 120  270, clip]{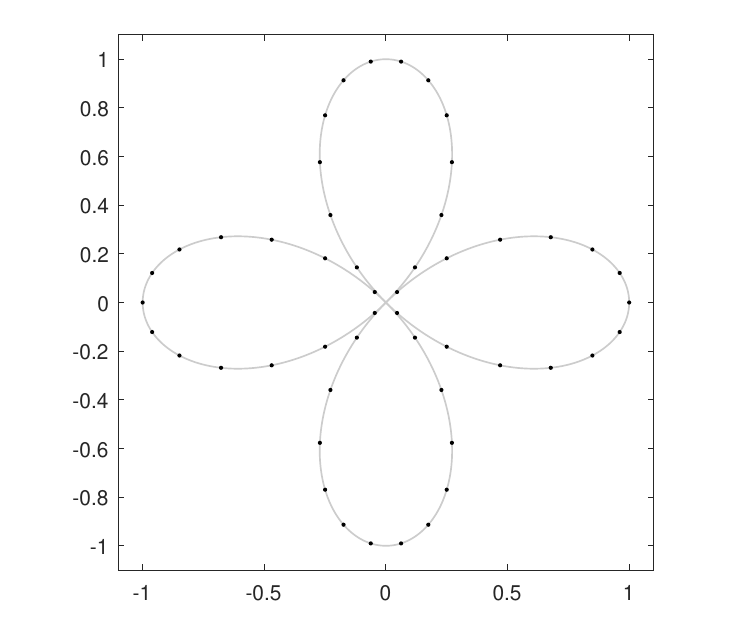}} 
		\subfigure[$t=0.07$ ]{\includegraphics[width=.31\textwidth,trim=120 270 120  270, clip]{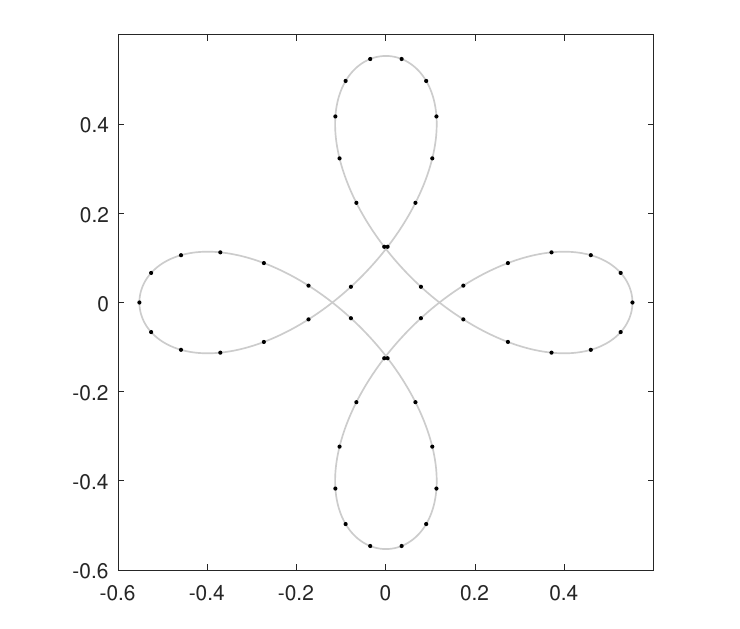}} 
		\subfigure[$t=0.0859$ ]{\includegraphics[width=.31\textwidth,trim=120 270 120  270, clip]{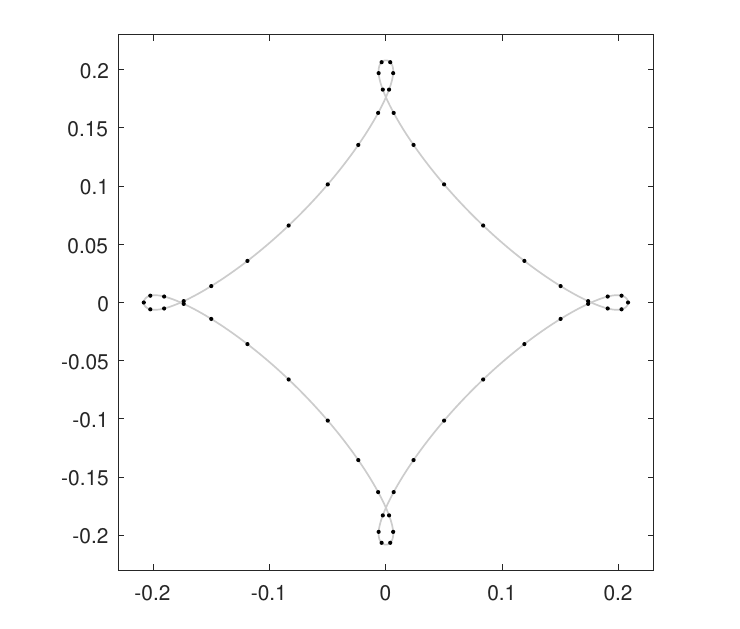}}\\
		\subfigure[$t=0.0861$ ]{\includegraphics[width=.31\textwidth,trim=120 270 120  270, clip]{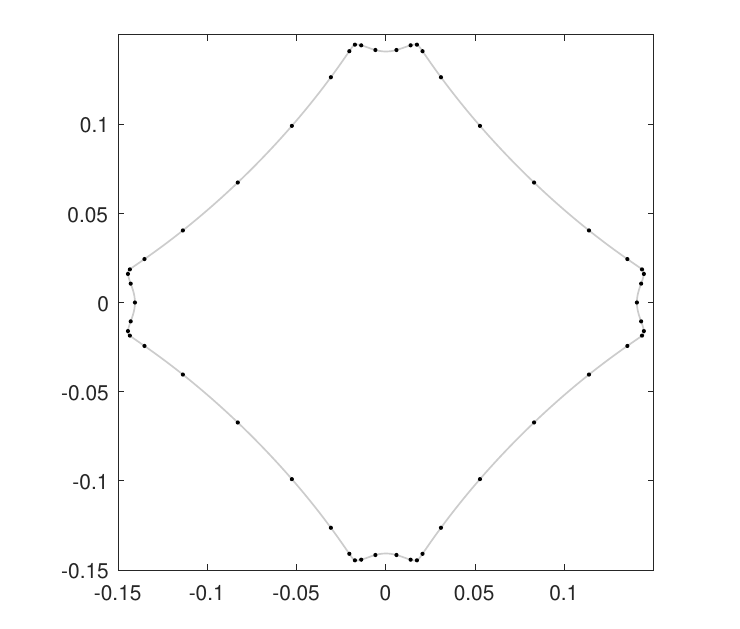}}
		\subfigure[$t=0.87$ ]{\includegraphics[width=.31\textwidth,trim=120 270 120  270, clip]{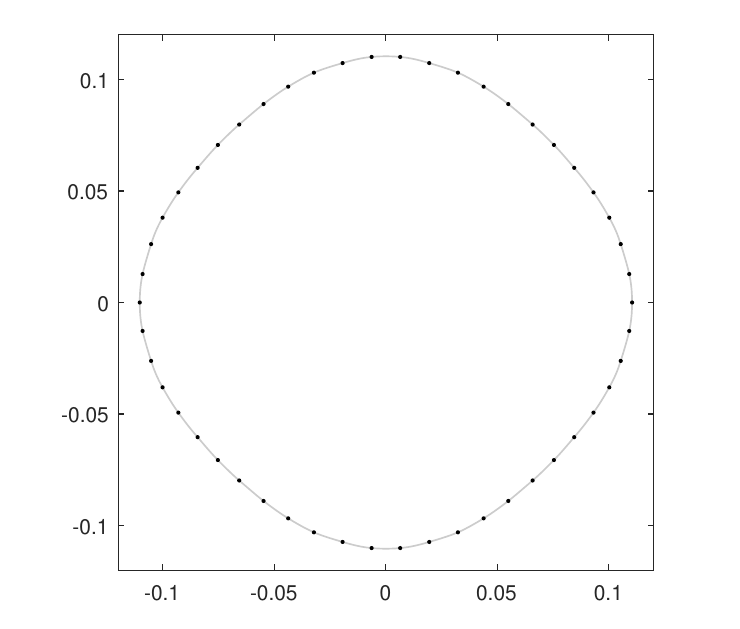}}
		\subfigure[$t=0.88$ ]{\includegraphics[width=.31\textwidth,trim=120 270 120  270, clip]{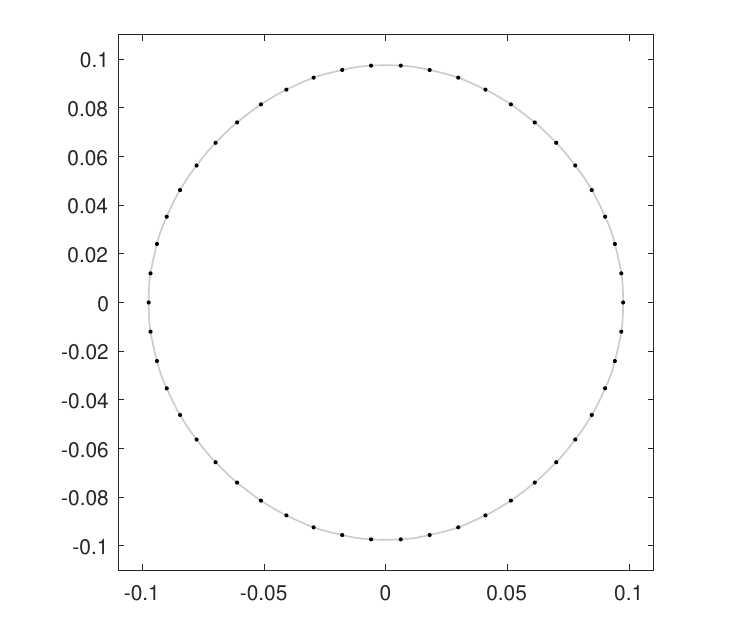}}	
		\caption{\Cref{test-2}: Numerical results of case  a) under $\Delta t=1\times10^{-4}$.}
		\label{fig-ex1}
	\end{figure}
	
	\begin{figure}
		\centering
		\subfigure[Initial curve of b)]{\includegraphics[width=.31\textwidth,trim=120 270 120  270, clip]{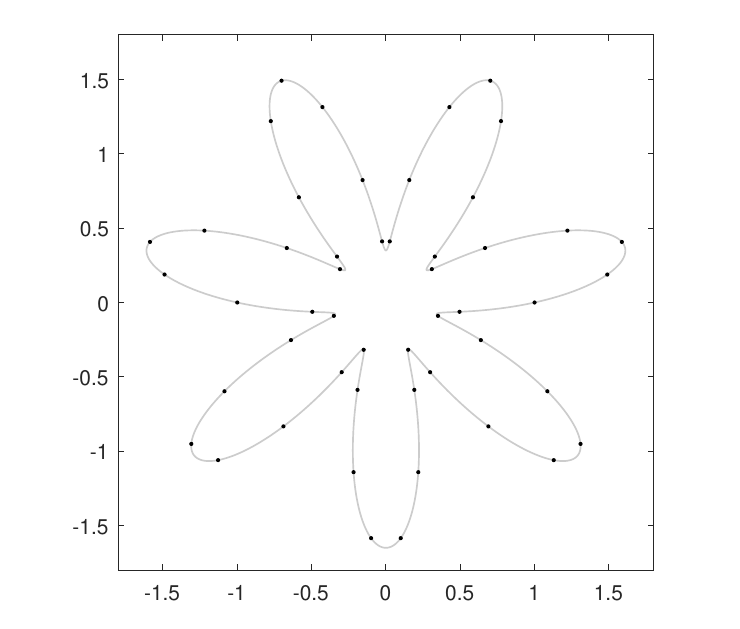}} 
		\subfigure[$t=0.04$ ]{\includegraphics[width=.31\textwidth,trim=120 270 120  270, clip]{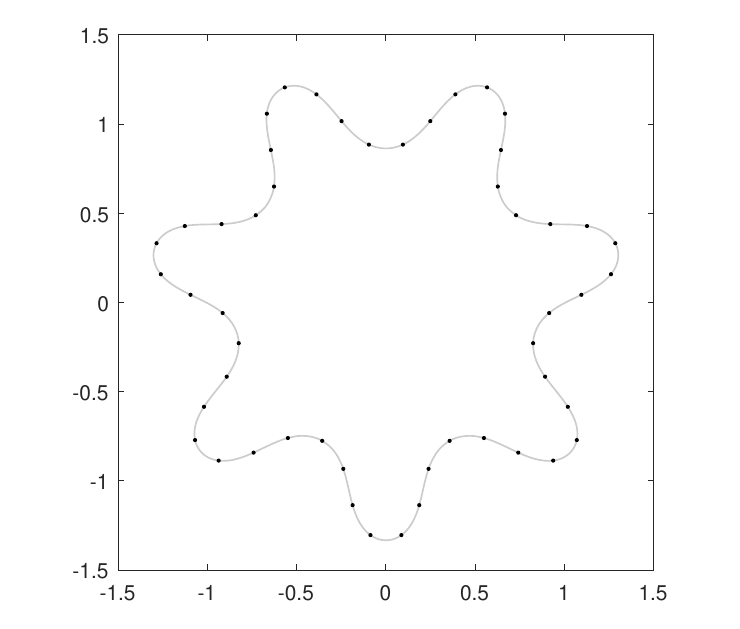}} 
		\subfigure[$t=0.4$ ]{\includegraphics[width=.31\textwidth,trim=120 270 120  270, clip]{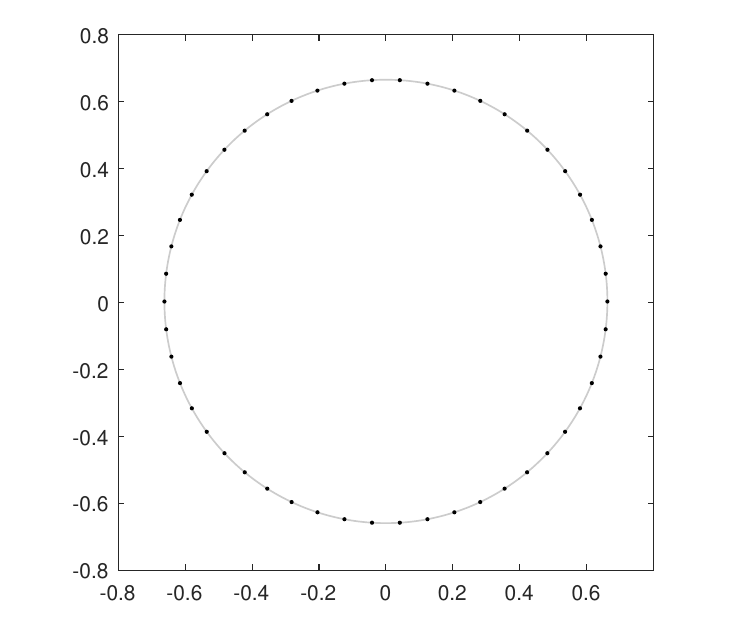}}\\
		\caption{\Cref{test-2}: Numerical results of case  b) under $\Delta t=1\times10^{-3}$.}
		\label{fig-ex2}
	\end{figure}
	
	\begin{figure}
		\centering
		\subfigure[Initial curve of c)]{\includegraphics[width=.31\textwidth,trim=120 270 120  270, clip]{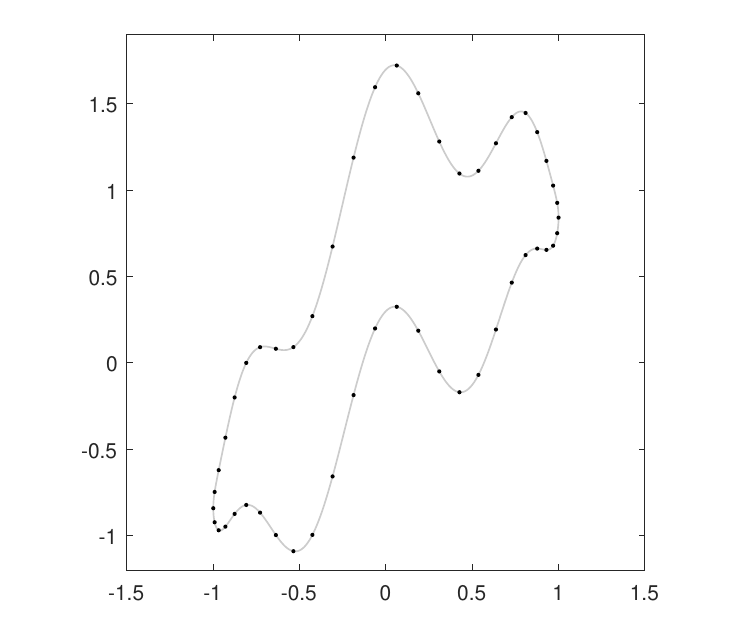}} 
		\subfigure[$t=0.1$ ]{\includegraphics[width=.31\textwidth,trim=120 270 120  270, clip]{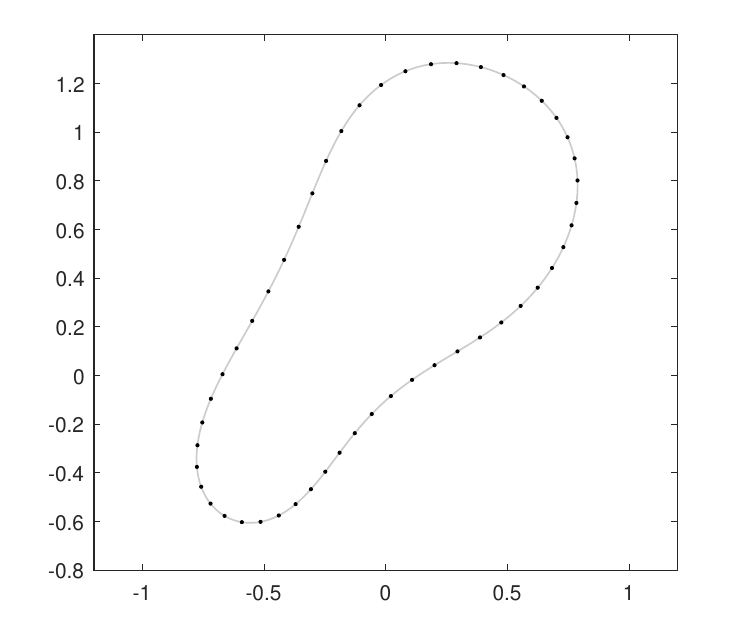}} 
		\subfigure[$t=0.345$ ]{\includegraphics[width=.31\textwidth,trim=120 270 120  270, clip]{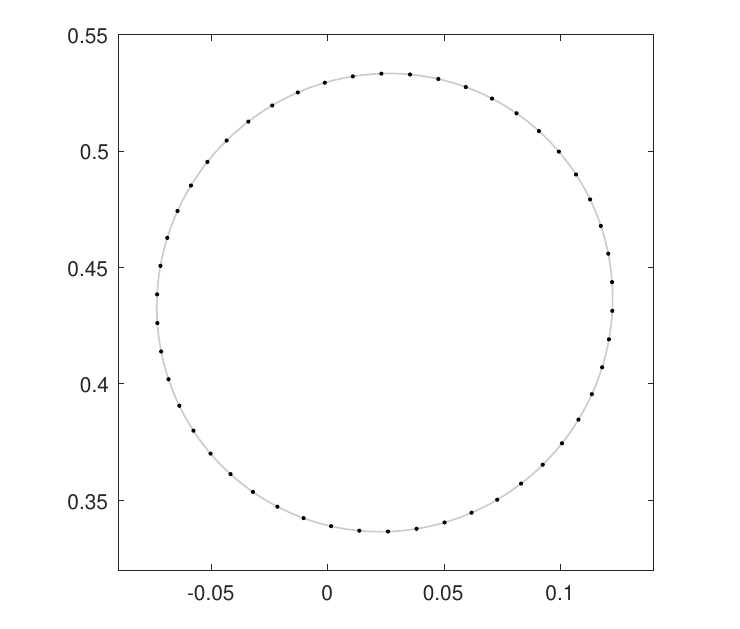}}\\
		\caption{\Cref{test-2}: Numerical results of case  c) under $\Delta t=5\times10^{-4}$.}
		\label{fig-ex3}
	\end{figure}
	
	We provide the discrete energy $|\Gamma_h^m|=\int_I |\partial_s\vX_h^m| ds$ in \Cref{fig-Eny}. For comparison we also include the results of elements of degree $k=1$ with $N=50\times 4=200$ nodes. One can see that the energy decreases monotonically for each case, and the results of $k=1$ and $k=4$ consist with each other well.
	
	\begin{figure}
		\centering
		\subfigure[case a)]{\includegraphics[width=.31\textwidth,trim=90 270 120  270, clip]{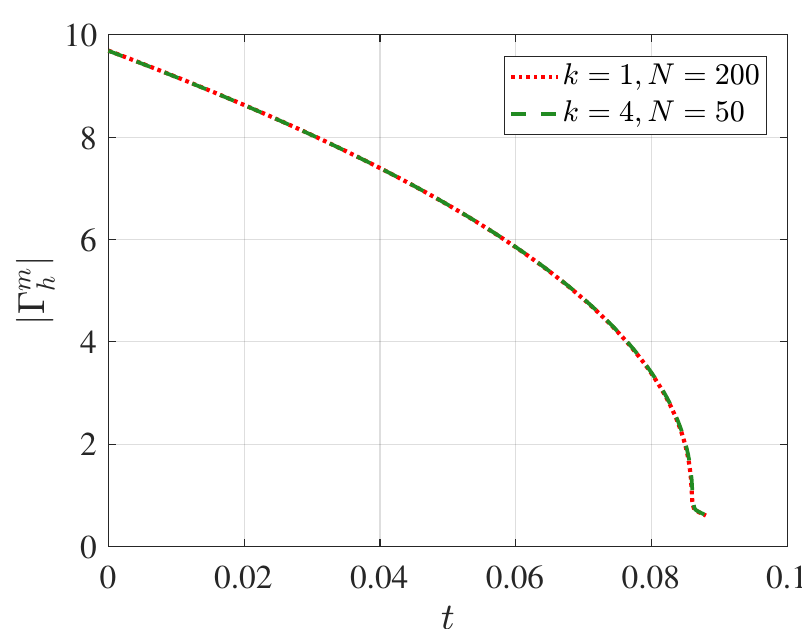}} 
		\subfigure[case b)]{\includegraphics[width=.31\textwidth,trim=90 270 120  270, clip]{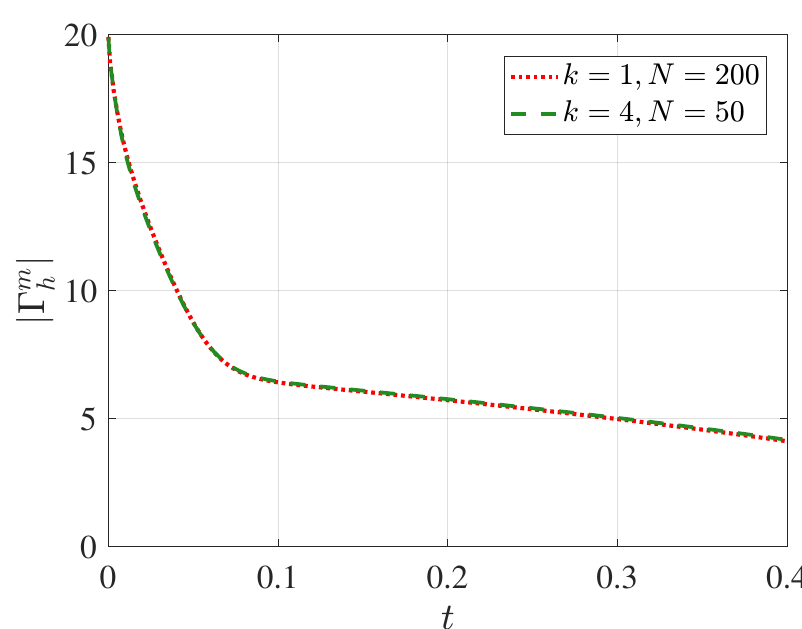}} 
		\subfigure[case c)]{\includegraphics[width=.31\textwidth,trim=90 270 120  270, clip]{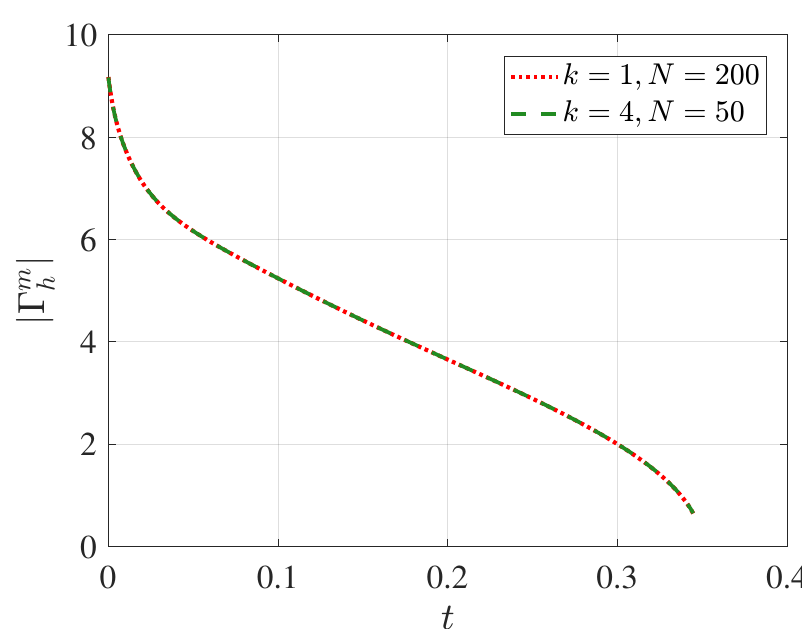}}
		\caption{\Cref{test-2}: $|\Gamma_h^m|$ with respect to time.}
		\label{fig-Eny}
	\end{figure}
	
	To further illustrate the distribution of $\big\{\vX_h^m(s_j)\big\}_{j=0}^{N-1}$ resembles that of $\{s_j\}_{j=0}^{N-1}$, we construct $\{s_j\}_{j=0}^{N-1}$ by $s_j=\pi+\pi \sin \frac{1}{2}(\theta_j-\pi) $ with $\theta_j=\frac{2\pi}{N}j$.  The nodes cluster at $s=0$, see \Cref{fig-ex3-nonuniform} (a). \Cref{fig-ex3-nonuniform} (b) - (c) present the numerical curves at $t=0.1$ and $t=0.345$, respectively, where the initial setting is given by case c) and the parameters are the same as before. In the two plots, we mark the position of $\vX_h^m(s_0)$ by red `$\circ$'. One can see that the nodes cluster at $\vX_h^m(s_0)$, resembling that of the reference mesh given in \Cref{fig-ex3-nonuniform} (a). To quantify the distribution of the nodes, we introduce the ratios $\{r_j^m\}_{j=1}^N$ between corresponding (curved) edges of  the reference mesh and $\Gamma_h^m$
	\begin{equation*}
		r_j^m=\frac{1}{|s_{j}-s_{j-1}|}\int_{s_{j-1}}^{s_j}|\partial_s\vX_h^m| \,ds. 
	\end{equation*}
	We then use $r^m=\frac{\max_{j}r_j^m}{\min_j r_j^m}$ to characterize the distribution of the nodes at $t=t_m$. \Cref{fig-ex3-ratio} presents $r^m$ under different $k$ and $N$ as the function of time. As expected, one can see $r^m$ approaches $1$  as time goes on.  In fact, in all the cases $r^{M}-1\approx 0.006$.
	\begin{figure}
		\centering
		\subfigure[Reference mesh]{\includegraphics[width=.31\textwidth,trim=90 270 120  270, clip]{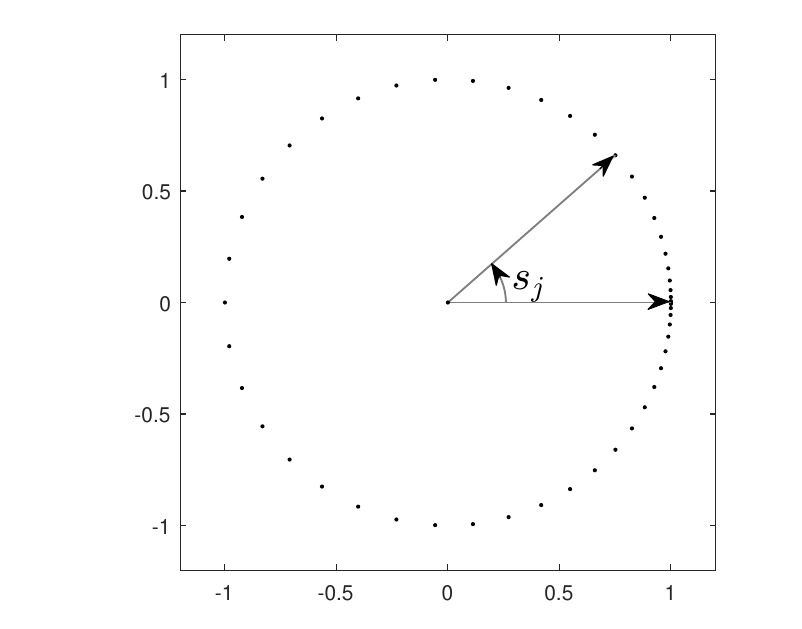}} 
		\subfigure[$t=0.1$ ]{\includegraphics[width=.31\textwidth,trim=120 270 120  270, clip]{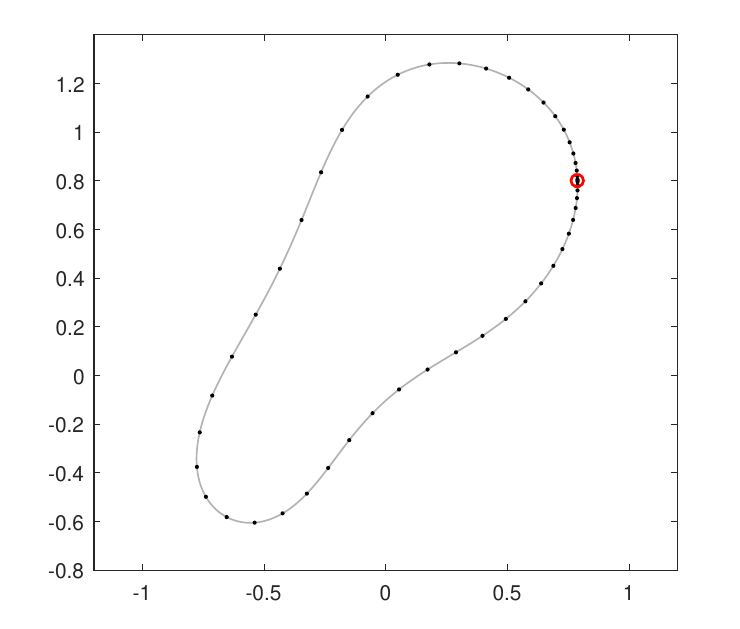}} 
		\subfigure[$t=0.345$ ]{\includegraphics[width=.31\textwidth,trim=120 270 120  270, clip]{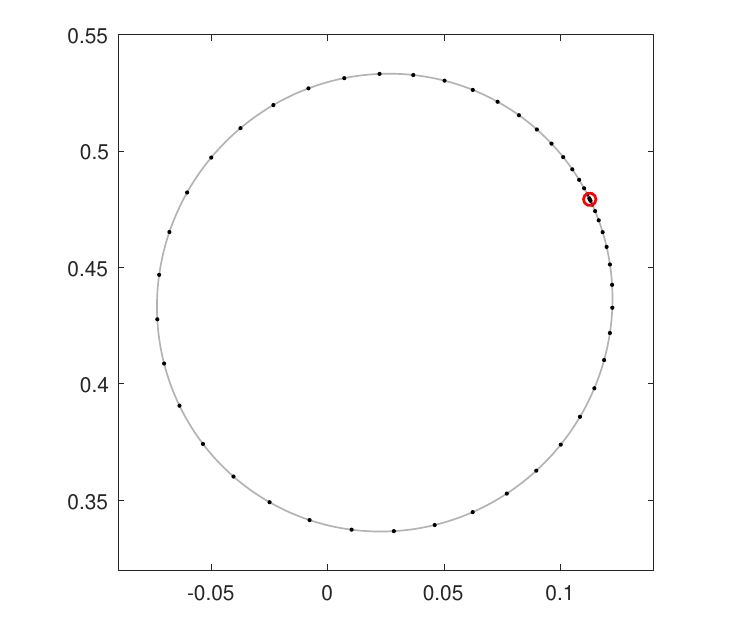}}\\
		\caption{\Cref{test-2}: Reference mesh (partition of $I$)  and numerical curves at different time of case  c) under $\Delta t=5\times10^{-4}$, $k=4$ and $N=50$. The position of $\vX_h^m(s_0)$ is marked by red ` $\circ$'.}
		\label{fig-ex3-nonuniform}
	\end{figure}
	\begin{figure}
		\centering
		\includegraphics[width=.5\textwidth,trim=90 270 120  270, clip]{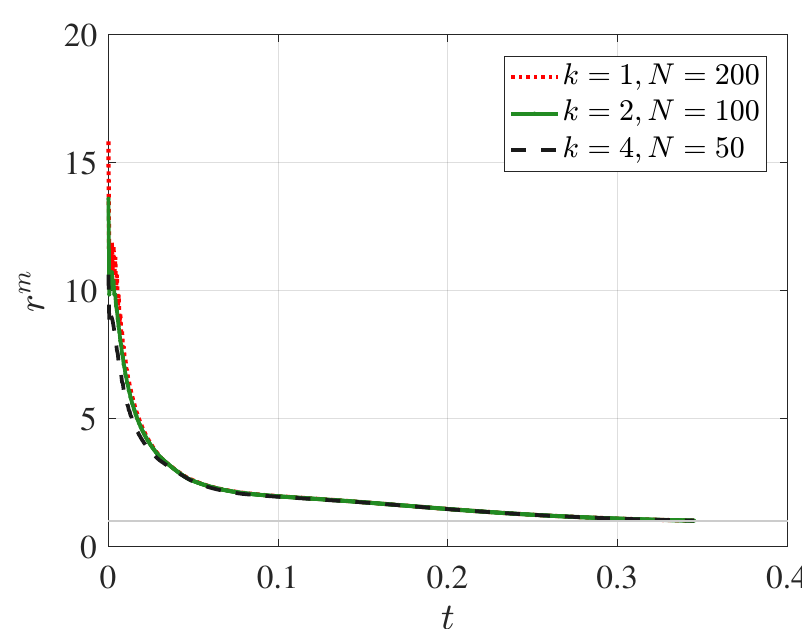}
		\caption{\Cref{test-2}: The mesh index $r^m$ at different time under different $k$ and $N$.}
		\label{fig-ex3-ratio}
	\end{figure}
\end{example}

\section{Conclusion}
We demonstrated the optimal $L^2$-error estimates for two fully discrete schemes of the curve shortening-DeTurck flow. Linearized Euler and Crank–Nicolson schemes were employed for temporal discretization, and the standard finite element method was applied in space. We also presented an extrinsic way to derive the curve shortening–DeTurck flow. This shows that the reformulated flow reduces both the perimeter and the harmonic energy; consequently, in the discrete setting, the resulting numerical curves distribute nodes in a manner resembling the reference mesh. This property facilitates the design of algorithms for adaptive meshing.

%

\bibliographystyle{abbrv}
\bibliography{References_GFs}
\end{document}